\documentclass[12pt,reqno]{amsart}
\usepackage[notref,notcite]{}
\usepackage{graphicx}
\usepackage{amssymb}
\usepackage{amsmath}
\usepackage{amsmath,color,pdfsync,comment}
\usepackage{float}
\usepackage{subcaption}

\usepackage[margin=1in]{geometry}

\newcommand{\R}{\mathbb{R}}
\newcommand{\D}{\mathbb{D}}
\newcommand{\Z}{\mathbb{Z}}

\newcommand{\PD}{\partial}

\newcommand{\Beq}{\begin{equation}}
    \newcommand{\Eeq}{\end{equation}}
\newcommand{\beq}{\begin{equation*}}
    \newcommand{\eeq}{\end{equation*}}
\newcommand{\bal}{\begin{align}}
    \newcommand{\eal}{\end{align}}

\let\RealPart\Re
\renewcommand{\Re}{\rangle_{e}}
\newcommand{\g}{\gamma}

\newcommand{\bp}{\begin{prob}}
    \newcommand{\ep}{\end{prob}}
\newcommand{\bpr}{\begin{proof}}
    \newcommand{\epr}{\end{proof}}

\newcommand{\bel}[1]{\begin{equation}\label{#1}}
    \newcommand{\ee}{\end{equation}}

\allowdisplaybreaks
\newtheorem{theorem}{Theorem}[section]

\newtheorem{proposition}[theorem]{Proposition}

\newtheorem{problem}{Problem}

\theoremstyle{definition}

\newcommand{\vp}{\varphi}

\newcommand{\Cb}{\mathbb{C}}

\newcommand{\Db}{\mathbb{D}}

\newcommand{\Sb}{\mathbb{S}}

\usepackage{hyperref}
\numberwithin{equation}{section}

\begin{document}


\title[Numerical solution of the two-dimensional Calder\'{o}n problem]{Numerical solution of the two-dimensional Calder\'{o}n problem
based on the Hilbert transform of a planar domain}
\author[S.\ Gohri and V.\ Sharafutdinov]{Sagar Gohri
and Vladimir Sharafutdinov
}
\address{TIFR Centre for Applicable Mathematics, Bangalore, India}
\email{sagar23@tifrbng.res.in}
\address{Sobolev Institute of mathematics. 4 Koptyug av., Novosibirsk, 630090, Russia}
\email{sharafut@list.ru}

\setcounter{section}{0}
\setcounter{page}{1}

\begin{abstract}
Let $(M,g)$ be a $C^{\infty}$ smooth compact connected Riemannian manifold with boundary $\partial M$. Consider the Dirichlet problem:
$\Delta_gu=0,\ u|_{\partial M}=f$ for $f\in C^{\infty}(\PD M)$.  The Diri\-chl\-et-to-Neumann (DN) operator
$\Lambda_g:C^\infty(\partial M)\longrightarrow C^\infty(\partial M)$ is defined by
$\Lambda_gf=\left.\frac{\partial u}{\partial\nu}\right|_{\partial M}$,
where $\nu$ is the unit outer normal to the boundary and $u$ is the unique solution to the Dirichlet problem.
Let $g_\partial$ be the Riemannian metric on $\partial M$ induced by $g$. The Calder\'{o}n problem is as follows: To what extent is $(M,g)$
determined by the data $(\partial M,g_\partial,\Lambda_g)$?
In the two-dimensional case the surface $(M,g)$ is determined by the DN data uniquely up to conformal equivalence.
Knowledge of the DN data is equivalent to knowledge of the Hilbert transform ${\mathcal H}_\Omega:C^\infty(\Gamma)\to C^\infty(\Gamma)$ on the
boundary curve $\Gamma=\partial\Omega$ of a planar domain $\Omega$. We study properties of the Hilbert transform. In particular, we obtain an
integral formula for ${\mathcal H}_\Omega$ for a simply connected $\Omega$ which generalizes the classical integral formula for the Hilbert
transform on the unit circle. This formula is the base of our algorithm for reconstructing a simply connected planar domain from the DN data.
Several numerical reconstructions are presented.

\noindent
Keywords: Dirichlet-to-Neumann map, Calder\'{o}n problem, Hilbert transform, numerical methods.
\end{abstract}

\maketitle

\section{Introduction}

Let $(M,g)$ be a smooth (synonymous with $C^{\infty}$ smooth) compact connected Riemannian manifold with non-empty boundary $\partial M$.
We denote by $g_\partial$ the Riemannian metric induced on $\PD M$ by $g$.
Let $\Delta_g=g^{ij}\nabla_{\!i}\nabla_{\!i}:C^\infty(M)\to C^\infty(M)$ be the Laplace-Beltrami operator of the metric $g$
where $(g^{ij})=(g_{ij})^{-1}$.
The Dirichlet-to-Neumann operator (DN-map),
$$
\Lambda_g:C^\infty(\partial M)\longrightarrow C^\infty(\partial M)
$$
is defined by
$\Lambda_gf=\left.\frac{\partial u}{\partial\nu}\right|_{\partial M}$,
where $\nu$ is the unit outer normal to the boundary and $u$ is the unique solution to the boundary value problem
$$
\Delta_gu=0\quad\mbox{in}\quad M,\quad u|_{\partial M}=f.
$$
As is well-known, $\Lambda_g$ is a first order pseudodifferential operator.

The geometric Calder\'on problem problem is posed as follows: To what extent is a Riemannian manifold $(M,g)$ determined by the data
$(\partial M,g_\partial,\Lambda_g)$?

The Calder\'on problem is the main mathematical tool of Electrical Impedance Tomography (EIT) where one tries to determine unisotropic
conductivity inside a domain from results of current-voltage measurements on the boundary of the domain. EIT has important applications
in medical diagnostics, see \cite{AGL} and \cite[Section 8]{Uh} for details.

As is well-known, one cannot recover the metric $g$ on $M$ uniquely from the above boundary information. If $\vp:M\to M$ is a diffeomorphism
fixing the boundary, then $g$ and the pull-back metric $g'=\vp^{*}g$ satisfy $g_{\PD}=g'_{\PD}$ and $\Lambda_g=\Lambda_{g'}$.
Additionally, in two-dimensions, the Calder\'on problem has a non-uniqueness due to the conformal invariance of the Laplace-Beltrami operator:
$\Delta_{\rho g}=\rho^{-1}\Delta_g$ for a positive function $\rho\in C^{\infty}(M)$. Hence if we take
$0<\rho\in C^{\infty}(M)$ such that $\rho|_{\PD M}\equiv 1$, then $\Lambda_{\rho g}=\Lambda_g$.

This paper focuses  on the two-dimensional Calder\'on problem.
The boundary of a Riemannian surface $(M,g)$ will be  denoted by $\Gamma=\partial M$. Instead of $g_\partial$, we use the arc-length $ds_g$
of the curve $\Gamma$ with respect to the metric $g$. For a compact Riemannian surface $(M,g)$ with non-empty boundary $\Gamma$, the DN-map
$$
\Lambda_g:C^\infty(\Gamma)\to C^\infty(\Gamma)
$$
is a non-negative self-adjoint operator with respect to the $L^2$-product
$$
(u,v)_{L^2(\Gamma)}=\int_\Gamma u\overline v\,ds_g\quad(u,v\in C^\infty(\Gamma)).
$$
The one-dimensional kernel of $\Lambda_g$ consists of constant functions while the range of $\Lambda_g$ coincides with the space
$$
C_0^\infty(\Gamma)=\Big\{f\in C^\infty(\Gamma)\mid\int_\Gamma f\,ds_g=0\Big\}
$$
of functions with zero mean value.

For a smooth map $\varphi:N\rightarrow N'$ between two manifolds, we define the linear operator
$\varphi^*:C^\infty(N')\rightarrow C^\infty(N)$ by $\varphi^* u=u\circ\varphi$.
We are now ready to state the following result.

\begin{theorem} \cite{LaU,Be}\label{Th1.1}
Let $(M_j,g_j)\ (j=1,2)$ be two compact Riemannian surfaces with non-empty boundaries $\Gamma_j=\partial M_j$ and let
$\varphi:(\Gamma_1,ds_{g_1})\rightarrow(\Gamma_2,ds_{g_2})$ be an isometry preserving the DN-map, i.e., such that the following diagram is
commutative:
$$
\begin{array}{ccc}
C^\infty(\Gamma_1)&\stackrel{\varphi^*}\longleftarrow&C^\infty(\Gamma_2)\\
\Lambda_{g_1}\downarrow&&\downarrow\Lambda_{g_2}\\
C^\infty(\Gamma_1)&\stackrel{\varphi^*}\longleftarrow&C^\infty(\Gamma_2).
\end{array}
$$
Then $\varphi$ extends to a diffeomorphism $\psi:M_1\rightarrow M_2$ such that $\psi|_{\Gamma_1}=\varphi$ and $\psi^*g_2=\rho g_1$ for some
function $0<\rho\in C^\infty(M_1)$ satisfying $\rho|_{\partial M_1}=1$.
\end{theorem}
To the best of our knowledge there are two different proofs of this theorem; one  by Lassas-Uhlmann \cite{LaU} and the other by Belishev
\cite{Be}; see also \cite{Shar}.
These results cover general compact Riemannian surfaces and the proofs are involved. However, for the case when the Riemannian surface is
simply connected,  an elementary proof of this result is presented in \cite{Sr2}.
It is based on the following fact. For a compact Riemannian surface $(M,g)$ with non-empty boundary, there exists a function
$0<\rho\in C^\infty(M)$ such that $\rho|_{\partial M}=1$ and $\rho g$ is a flat metric, i.e., its Gaussian curvature is identically zero.
This implies, in the case of a simply connected $M$, that $(M,\rho g)$ can be isometrically immersed into the Euclidean plane.
This reduces Theorem \ref{Th1.1} to the partial case when both surfaces $M_1$ and $M_2$ are simply connected, possibly multi-sheeted planar
domains and metrics $g_1$ and $g_2$ coincide with the standard Euclidean metric of ${\R}^2$. After this reduction, the theorem is easily
proved by using basic properties of conformal maps.

The current paper deals with the numerical implementation of the aforementioned result \cite{Sr2}, that is, for simply connected compact
Riemannian surfaces. This article is a follow-up of the previous work of Sharafutdinov and  Storozhuk \cite{SharStor} where they considered
the numerical implementation of the reconstruction of Riemannian surfaces $(M,g)$ close to $(\Db,e)$, where $\Db$ is the unit disc in $\Cb$
and $e$ is the Euclidean metric. In contrast, the current work deals with the reconstruction of general simply connected possibly multi-sheeted
domains from the knowledge of the Dirichlet-to-Neumann map.

The reconstruction algorithm in \cite{SharStor} is an iterative scheme: The function $\gamma$ parameterizing the boundary $\Gamma$ of an unknown
domain $\Omega$ satisfies an infinite system of quadratic equations in terms of its Fourier coefficients $\widehat{\gamma}_{n}$.
This infinite system is replaced by a finite system of quadratic equations, which then is solved via an
iteration method. A choice of an initial approximation to the domain to be recovered is crucial. Due to this reason, only those domains
$\Omega$ close enough to the unit disk $\Db$ were considered.

The reconstruction algorithm in this paper is not based on solving a system of quadratic equations as was the case in \cite{SharStor}.
Instead, we use an integral formula for the Hilbert transform $H_{\gamma}$; see \eqref{5.1} for the definition of the Hilbert transform.
Logically, our reconstruction algorithm is more complicated than that of \cite{SharStor} but numerically it is much easier since it
actually consists of a chain of explicit formulas. In particular, our algorithm does not involve any iterative procedure. Unlike \cite{SharStor},
our algorithm works for simply connected multi-sheeted domains as well.

The rest of the paper is organized as follows. In Section \ref{Preliminaries}, we give the requisite theoretical preliminaries,
specifically, the result from \cite{Sr2,Shar} required for the numerical reconstruction  algorithm.

Our main theoretical results are presented in Section \ref{Sec:HT}. The classical Hilbert transform on the unit circle plays an important role
in many inverse problems, see for example \cite{PU,BS,JS2}. To our knowledge, the Hilbert transform ${\mathcal H}_\Omega$
on the boundary curve of an arbitrary planar domain $\Omega$
was first introduced in the recent work \cite{SharStor}. Unexpectedly, ${\mathcal H}_\Omega$ inherits some good properties of the classical
Hilbert transform, see Theorem \ref{Th5.1} below. In particular, we prove an integral formula for ${\mathcal H}_\Omega$ in the case of a
simply connected domain $\Omega$ (Theorem \ref{Th5.3} below) which generalizes the well known integral formula for the classical Hilbert
transform. This formula serves as the basis of our algorithm for numerical solution of the Calder\'{o}n problem
presented in Section 4. Several numerical examples are presented in Section \ref{numerical reconstructios}.

We are grateful to Venky Krishnan for his help with preparing the paper.

\section{Preliminaries}\label{Preliminaries}
In this section, we summarize some results from \cite{Sr2,Shar}. We recall that a Riemannian metric $g$  on a two-dimensional
manifold is said to be a flat metric if its Gaussian curvature is identically zero.

\begin{proposition} {\cite[Lemma 2.1]{Sr2},\cite[Proposition 3.1]{Shar}}\label{P2.1}
Let $(M,g)$ be a compact Riemannian surface with non-empty boundary. There exists a unique positive function $\rho\in C^\infty(M)$
satisfying the boundary condition $\rho|_{\partial M}=1$ and such that $\rho g$ is a flat metric.
\end{proposition}

\begin{proposition} {\cite[Proposition 2.2]{Sr2}}\label{P2.2}
A simply connected flat Riemannian surface $(M,g)$ admits a locally isometric immersion into the Euclidean plane, i.e., there exists a smooth map
\begin{equation}
I:M\to{\R}^2
                                \label{2.4}
\end{equation}
such that, for every point $x\in M$, the differential $d_xI:T_xM\to{\R}^2$ is an isometry of the tangent space $T_xM$ furnished with the dot
product $\langle\cdot,\cdot\rangle_g$ onto ${\R}^2$ furnished with the standard Euclidean dot product.
If additionally the surface $M$ is oriented and the map \eqref{2.4} is assumed to transform the chosen orientation of $M$ to the standard
orientation of ${\R}^2$, then the immersion $I$ is unique up to the composition with a shift and rotation of the plane.
\end{proposition}

\begin{figure}[h]
\center{\includegraphics[width=3.0in,keepaspectratio]{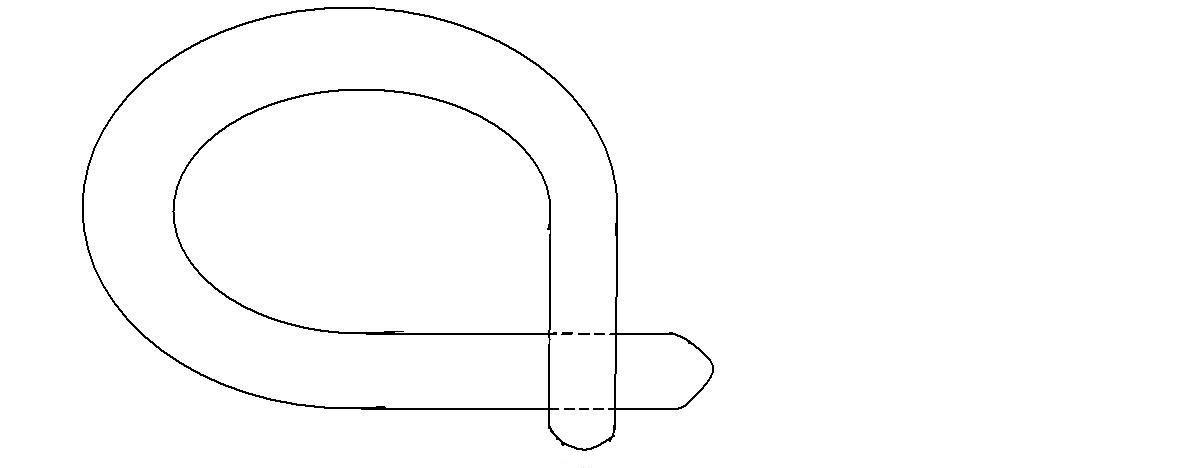}}
\caption{Example of multi-sheeted domain}
    \label{multisheet}
\end{figure}

Based on the above two results, we consider  the Calder\'{o}n problem for a compact simply connected flat Riemannian surface $(M,g)$ with a
non-empty boundary. In such a case $(M,g)$ can be identified with the closed simply connected, probably multi-sheeted,
planar domain
$\Omega=I(M)\subset{\mathbb C}$ bounded by the smooth closed curve $\Gamma=I(\partial M)$ and
furnished with the standard Euclidean metric.

An example of multi-sheeted domain is presented on Figure 1. For such domains, the definition of spaces $C^\infty(\Omega)$ and
$C^\infty(\Gamma)$ should be slightly modified. Roughly speaking, a function $f\in C^\infty(\Omega)$ can take different values at different
copies of a multiple point $x\in\Omega$. We omit obvious details of the modification.

Taking into account the identification based on \eqref{2.4}, we let $\Omega\subset \mathbb{R}^2$ be a closed simply connected, probably
multi-sheeted, domain bounded
by a smooth closed curve $\Gamma$. We assume $\Gamma$ is given the usual anticlockwise orientation. Recall that the metric on $\Omega$ is the
standard Euclidean metric. Let
$$
\Lambda_\Omega:C^\infty(\Gamma)\to C^\infty(\Gamma)
$$
be the DN-map of $\Omega$.
We emphasize that $\Gamma$ is unknown to us with the only known information being its length, which we assume to be $2\pi$. Let
${\mathbb S}=\{e^{i s}\mid  s\in{\R}\}$ be the unit circle oriented in the anticlockwise direction. For a function $u\in C^\infty({\mathbb S})$,
we write $u(s)$ instead of $u(e^{i s})$.
We parameterize the curve $\Gamma$ by the arc-length $s$ measured in the positive direction from an initial point.
The parametrization determines a smooth $2\pi$-periodic function
$\gamma:{\R}\to {\mathbb C}$ satisfying
$$
\Big|\frac{d\gamma( s)}{d s}\Big|=1.
$$
Treating $\g$ as the diffeomorphism $\Sb \to \Gamma$ preserving the arc-length and orientation, we define
\begin{equation}
\Lambda_\gamma=\gamma^* \Lambda_\Omega\gamma^* {}^{-1}:C^\infty({\mathbb S})\to C^\infty({\mathbb S}).
                                \label{3.4}
\end{equation}
We call $ \Lambda_\gamma$ {\it the DN-map of the domain $\Omega$ in the natural parametrization}. We are now ready to state the main inverse
problem, which we repeat from \cite{SharStor} for the sake of completeness.

\begin{problem} \label{Pr3.1}
Let
$\Omega\subset{\mathbb C}$ be a closed simply connected, probably multi-sheeted, domain  bounded by a closed smooth curve
$\Gamma=\partial\Omega$ of length $2\pi$.
 Given the operator \eqref{3.4}, the goal is to recover the domain
$\Omega$ up to a shift and rotation.
\end{problem}

\section{The Hilbert transform}\label{Sec:HT}
Again, let $\Omega\subset{\mathbb C}={\R}^2$ be a closed simply connected, probably
multi-sheeted, domain  bounded by a closed smooth curve $\Gamma=\partial\Omega$
of length $2\pi$ and let $\gamma:{\mathbb S}\to\Gamma$ be a diffeomorphism preserving the arc-length and orientation.

Let $D=-i\frac{d}{ds}:C^\infty(\Gamma)\to C^\infty(\Gamma)$ be the differentiation with respect to the arc-length in positive direction
($i$ is the imaginary unit). Recall that $C_0^\infty(\Gamma)=\{f\in C^\infty(\Gamma)\mid\int_\Gamma f\,ds=0\}$. The restriction
$D:C_0^\infty(\Gamma)\to C_0^\infty(\Gamma)$ is the isomorphism. Let $D^{-1}:C_0^\infty(\Gamma)\to C_0^\infty(\Gamma)$ be the inverse isomorphism.
Thus, for $f\in C_0^\infty(\Gamma)$, $D^{-1}f$ is the unique anti-derivative with zero mean value.

The operator
\begin{equation}
{\mathcal H}_\Gamma=D^{-1} \Lambda_\Omega:C^\infty(\Gamma)\to C^\infty(\Gamma)
                                \label{5.0}
\end{equation}
is well defined since the range of $\Lambda_\Omega$ coincides with $C_0^\infty(\Gamma)$.
It will be called {\it the Hilbert transform on the curve $\Gamma$} while the operator
\begin{equation}
{\mathcal H}_\gamma=\gamma^* {\mathcal H}_\Gamma\gamma^{* -1}:C^\infty({\mathbb S})\to C^\infty({\mathbb S})
                                \label{5.1}
\end{equation}
will be called {\it the Hilbert transform of the curve $\Gamma$ in the natural parametrization}. It is a zero order pseudodifferential
operator.

We denote by ${\mathcal A}(\Omega)$ the subspace of $C^\infty(\Omega)$ consisting of functions
$f\in C^\infty(\Omega)$ such that the restriction of $f$ to the interior $\Omega\setminus\Gamma$ is a holomorphic function.
(The space ${\mathcal A}(\Omega)$ was denoted by ${\mathcal A}^\infty(\Omega)$ in \cite{Be} and by
${\mathcal A}^\infty(M,g)$ in \cite{Shar}. It has a Banach algebra structure that is not used in the current work.)
Elements of ${\mathcal A}(\Omega)$ will be called {\it holomorphic functions} on $\Omega$.
We also introduce the space
$$
{\mathcal A}(\Gamma)=\{w|_{\Gamma}\mid w\in{\mathcal A}(\Omega)\}
$$
of {\it boundary traces of holomorphic functions}.

\begin{proposition}\cite[Lemma 1]{Be}, \cite[Prop. 6.1]{Shar} \label{P4.1}
Let $\Omega\subset{\mathbb C}$ be a closed simply connected, probably
multi-sheeted, domain  bounded by a closed smooth curve $\Gamma=\partial\Omega$. Then

{\rm (1)} The equality
\begin{equation}
\Big(1-(\Lambda_\Omega D^{-1})^2\Big)Df=0,
                                \label{4.1}
\end{equation}
holds for every $f\in C^\infty(\Gamma)$.

{\rm (2)}
For $f,g\in C^\infty(\Gamma)$, the function $f+ig$ belongs to ${\mathcal A}(\Gamma)$ if and only if
\begin{equation}
g=-iD^{-1}\Lambda_\Omega f+c
                                \label{4.2}
\end{equation}
with a constant $c$.
\end{proposition}

Observe that
$$
D=\gamma^* D\gamma^* {}^{-1},\quad
D^{-1}=\gamma^* D^{-1}\gamma^* {}^{-1}.
$$
On the right-hand side of the last equality, $D^{-1}:C^\infty_0({\mathbb S})\to C^\infty_0({\mathbb S})$ is the inverse operator of
$$
D=-i\frac{d}{d s}: C^\infty_0({\mathbb S})\to
C^\infty_0({\mathbb S})=\Big\{f\in  C^\infty({\mathbb S})\mid\int_0^{2\pi}f(s)\,ds=0\Big\}.
$$

Since $D^{-1}Df=f+\mbox{const}$ and constants lie in the kernel of $\Lambda_\Omega$, \eqref{4.1} can be written as follows:
$Df=\Lambda_\Omega D^{-1}\Lambda_\Omega f$ for $f\in C^\infty(\Gamma)$. This implies
\begin{equation}
Df=\Lambda_\gamma D^{-1}\Lambda_\gamma f\quad\mbox{for every} \ f\in C^\infty({\mathbb S}).
                                \label{4.4}
\end{equation}
Write \eqref{4.2} in the form
$$
f+ig=f+D^{-1}\Lambda_\Omega f+ic=(f+ic)+D^{-1}\Lambda_\Omega(f+ic).
$$
Setting $f'=f+ic$, we have $f+ig=f'+D^{-1}\Lambda_\Omega f'$.
The second statement  of Propositions \ref{P4.1} becomes
$$
{\mathcal A}(\Gamma)=\{f+D^{-1}\Lambda_\Omega f\mid f\in C^\infty(\Gamma)\}.
$$
The operators $D^{-1}$ and $\gamma^* $ commute since $\gamma$ preserves the arc-length. Therefore the previous formula implies
\begin{equation}
{\mathcal A}(\gamma)=\{f+D^{-1}\Lambda_\gamma f\mid f\in C^\infty({\mathbb S})\}.
                                \label{4.6}
\end{equation}
Finally, we define
$$
{\mathcal A}(\gamma)=\{\gamma^* f\mid f\in{\mathcal A}(\Gamma)\}.
$$
Main properties of the Hilbert transform ${\mathcal H}_\gamma$ are listed in the following theorem.

\begin{theorem} \label{Th5.1}
Let $\Omega\subset{\mathbb C}$ be a simply connected, probably
multi-sheeted, domain  bounded by a smooth closed curve $\Gamma$ of length $2\pi$ and let
$\gamma:{\mathbb S}\to\Gamma$ be a diffeomorphism preserving the arc-length and orientation.
The Hilbert transform \eqref{5.1} has the following properties.

{\rm (1)} The kernel of ${\mathcal H}_\gamma$ is the one-dimensional space consisting of constant functions.

{\rm (2)} The restriction of ${\mathcal H}_\gamma$ to $C_0^\infty({\mathbb S})=\{f\in C^\infty({\mathbb S})\mid\int_0^{2\pi}f\,ds=0\}$
$$
{\mathcal H}_\gamma|_{C_0^\infty({\mathbb S})}:C_0^\infty({\mathbb S})\to C_0^\infty({\mathbb S})
$$
satisfies
\begin{equation}
({\mathcal H}_\gamma|_{C_0^\infty({\mathbb S})})^2=1,
                                \label{5.2}
\end{equation}
where 1 is the identity operator.

{\rm (3)} Eigenvalues of ${\mathcal H}_\gamma|_{C_0^\infty({\mathbb S})}$ are $+1$ and $-1$.
The eigenspace of ${\mathcal H}_\gamma$ corresponding to the eigenvalue $+1$ is
${\mathcal A}_0(\gamma)={\mathcal A}(\gamma)\cap C_0^\infty({\mathbb S})$.
The eigenspace of ${\mathcal H}_\gamma$ corresponding to the eigenvalue $-1$ is
$\overline{{\mathcal A}_0(\gamma)}=\{\overline f\mid f\in{\mathcal A}_0(\gamma)\}$. The following decomposition holds
\begin{equation}
C_0^\infty({\mathbb S})={\mathcal A}_0(\gamma)\oplus\overline{{\mathcal A}_0(\gamma)}.
                                \label{5.3}
\end{equation}
Summands on the right-hand side of \eqref{5.3} are orthogonal with respect to the sesquilinear form
\begin{equation}
\langle f,g\rangle_\gamma=\int\limits_0^{2\pi}f(s)\overline{g(s)}\,\frac{d\gamma}{ds}\,ds.
                                \label{5.4}
\end{equation}
\end{theorem}

{\bf Remark.} The form \eqref{5.4} is not Hermitian, i.e., $\langle f,g\rangle_\gamma\neq\overline{\langle g,f\rangle_\gamma}$ since the
factor $\frac{d\gamma}{ds}$ is not real.

\begin{proof}
The kernel of $\Lambda_\Omega$ consists of constant functions while its range coincides with $C_0^\infty(\Gamma)$, the operator
$D^{-1}:C_0^\infty(\Gamma)\to C_0^\infty(\Gamma)$ is injective. Therefore the kernel of ${\mathcal H}_\Gamma=D^{-1}\Lambda_\Omega$ consists of
constant functions. This implies statement (1).

Let us reproduce the equation \eqref{4.1} in the form
$$
\Lambda_\Omega D^{-1}\Lambda_\Omega D^{-1}Df=Df\quad(f\in C^\infty(\Gamma)).
$$
Since $D^{-1}Df=f$ for $f\in C_0^\infty(\Gamma)$, this can be written as
$$
\Lambda_\Omega D^{-1}\Lambda_\Omega f=Df\quad(f\in C_0^\infty(\Gamma)).
$$
Applying the operator $D^{-1}$ to this equation and using $D^{-1}Df=f$ again, we obtain
$$
(D^{-1}\Lambda_\Omega)^2 f=f\quad(f\in C_0^\infty(\Gamma)).
$$
This can be written in terms of ${\mathcal H}_\Gamma$ as follows:
${\mathcal H}_\Gamma^2 f=f$ for $f\in C_0^\infty(\Gamma)).$
This implies statement (2).

We compute the square of $1+D^{-1} \Lambda_\gamma$
\begin{equation}
(1+D^{-1} \Lambda_\gamma)^2=1+2D^{-1} \Lambda_\gamma+D^{-1}(\Lambda_\gamma D^{-1} \Lambda_\gamma).
                                \label{5.5}
\end{equation}
By \eqref{4.4}, $\Lambda_\gamma D^{-1} \Lambda_\gamma=D$.
Substitute this expression into \eqref{5.5}
\begin{equation}
(1+D^{-1} \Lambda_\gamma)^2=1+2D^{-1} \Lambda_\gamma+D^{-1}D.
                                \label{5.6}
\end{equation}

Let $g\in{\mathcal A}(\gamma)$. By \eqref{4.6}, $g$ can be represented in the form
\begin{equation}
g=(1+D^{-1} \Lambda_\gamma)f
                                \label{5.7}
\end{equation}
with some $f\in C^\infty({\mathbb S})$. Hence
$
(1+D^{-1} \Lambda_\gamma)g=(1+D^{-1} \Lambda_\gamma)^2f.
$
With the help of \eqref{5.6}, this gives
\begin{equation}
(1+D^{-1}\Lambda_\gamma)g=f+2D^{-1} \Lambda_\gamma f+D^{-1}Df.
                                \label{5.8}
\end{equation}
If $g$ has the zero mean value, then $f$ has also the zero mean value, as is seen from \eqref{5.8}. For such functions,
$D^{-1}Df=f$ and the formula \eqref{5.8} takes the form
$$
(1+D^{-1}\Lambda_\gamma)g=2(1+D^{-1} \Lambda_\gamma)f.
$$
Together with \eqref{5.7}, this gives $(1+D^{-1}\Lambda_\gamma)g=2g$, i.e., ${\mathcal H}_\gamma g=D^{-1}\Lambda_\gamma g=g$. Thus, every
function $g\in{\mathcal A}_0(\gamma)$ is an eigenfunction of ${\mathcal H}_\gamma$ with the eigenvalue +1.

Conversely, assume a function $g\in C^\infty({\mathbb S})$ to satisfy ${\mathcal H}_\gamma g=D^{-1}\Lambda_\gamma g=g$. Then, first of all,
$g\in C_0^\infty({\mathbb S})$ since the range of $\Lambda_\gamma$ coincides with $C_0^\infty({\mathbb S})$ and $D^{-1}$ maps
$C_0^\infty({\mathbb S})$ onto itself. We write the equation $D^{-1}\Lambda_\gamma g=g$ in the form
$
2g=(1+D^{-1} \Lambda_\gamma)g.
$
By \eqref{4.6}, this means that $2g\in{\mathcal A}(\gamma)$. Therefore $g\in{\mathcal A}_0(\gamma)$.

We have thus proven that the eigenspace of ${\mathcal H}_\gamma$ corresponding to the eigenvalue $+1$ is
${\mathcal A}_0(\gamma)$. Using the identity $\overline{{\mathcal H}_\gamma f}=-{\mathcal H}_\gamma\overline f$, we obtain the dual statement:
the eigenspace of ${\mathcal H}_\gamma$ corresponding to the eigenvalue $-1$ is
$\overline{{\mathcal A}_0(\gamma)}$.

Every function $f\in C_0^\infty({\mathbb S})$ can be represented in the form
$$
f=\frac{1}{2}(f+{\mathcal H}_\gamma f)+\frac{1}{2}(f-{\mathcal H}_\gamma f).
$$
By \eqref{5.2}, the first summand on the right-hand side is the eigenvector of ${\mathcal H}_\gamma$ with the eigenvalue +1 and the second
summand is the eigenvector of ${\mathcal H}_\gamma$ with the eigenvalue -1. As we have already proven, this means that the first (second)
summand belongs to ${\mathcal A}_0(\gamma)$ (belongs to $\overline{{\mathcal A}_0(\gamma)}$). This proves \eqref{5.3}.

Finally, we prove that ${\mathcal A}_0(\gamma)$ and $\overline{{\mathcal A}_0(\gamma)}$ are orthogonal with respect to the form
\eqref{5.4}. Given two functions $f_1, f_2\in{\mathcal A}_0(\gamma)$, we have
$$
\langle f_1,\overline{f_2}\rangle_\gamma=\int\limits_0^{2\pi}f_1(s)f_2(s)\,\frac{d\gamma}{ds}\,ds.
$$
Changing the integration variable by $z=\gamma(s)$, we obtain
$$
\langle f_1,\overline{f_2}\rangle_\gamma=\int\limits_\Gamma g_1(z)g_2(z)\,dz,
$$
where $g_i=\gamma^{* -1}f_i\in{\mathcal A}_0(\Gamma)\ (i=1,2)$. The product $g_1g_2$ extends to a function belonging to
${\mathcal A}(\Omega)$. Therefore the last integral is equal to zero by the Cauchy theorem.
\end{proof}

{\bf Remark.} Theorem \ref{Th5.1} can be generalized to
multi-connected domains. But the proof should be slightly modified
since formulas \eqref{4.1}--\eqref{4.2} are replaced with more
complicated formulas when $\Gamma$ has several connection
components, compare with Proposition 6.1 of \cite{Shar}.

In the case when $\Gamma={\mathbb S}$ and $\gamma$ is the identity map, the Hilbert transform ${\mathcal H}_{\mathbb S}$ is the integral
operator \cite[formula (5.4)]{JS0}: for $u\in C^\infty({\mathbb S})$,
\begin{equation}
({\mathcal H}_{\mathbb S}u)(e^{i\theta})=\frac{1}{2\pi i}\int\limits_0^{2\pi}\cot\frac{t-\theta}{2}\,u(e^{it})\,dt,
                                \label{5.10}
\end{equation}
where the integral is understood in the sense of the principle value at $t=\theta$. We are going to generalize this formula to the general case.

Let $\Omega\subset{\mathbb C}$ be a simply connected, probably multi-sheeted, domain bounded by a smooth closed curve $\Gamma$ of length $2\pi$.
As before, we choose a diffeomorphism $\gamma:{\mathbb S}\to\Gamma$ preserving arc-length and orientation.

Recall that ${\D}=\{z\in{\mathbb C}\mid|z|\le1\}$ is the unit disk. By the Riemann theorem, there exists a biholomorphism
\begin{equation}
\Phi:{\D}\to\Omega.
                                \label{5.11}
\end{equation}
The function $\Phi:{\D}\to{\mathbb C}$ belongs to $C^\infty(\D)$, is holomorphic in ${\D}\setminus{\mathbb S}$, and the complex
derivative $\Phi'(z)$ does not vanish in $\D$.
The restriction $\varphi:{\mathbb S}\to\Gamma$ of $\Phi$ to the unit circle is a diffeomorphism preserving orientation.
The biholomorphism is not unique: $\widetilde\Phi=\Phi\circ\Psi$ is also such a biholomorphism for any conformal transformation
$\Psi$ of the unit disk.
To eliminate the non-uniqueness, we fix some $s_0\in(0,\pi)$ and assume that
\begin{equation}
\varphi(1)=1,\quad\varphi(e^{is_0})=\gamma(e^{is_0}),\quad\varphi(e^{-is_0})=\gamma(e^{-is_0}).
                                \label{5.12}
\end{equation}
By the well known {\it boundary correspondence principle} for conformal maps, there exists a unique biholomorphism
$\Phi:{\D}\to\Omega$ satisfying \eqref{5.12}.

We define the diffeomorphism
$\psi=\varphi^{-1}\circ\gamma:{\mathbb S}\to{\mathbb S}$, see Figure \ref{Fig12}.

\begin{figure}[h]
\center{\includegraphics[scale=0.65,keepaspectratio]{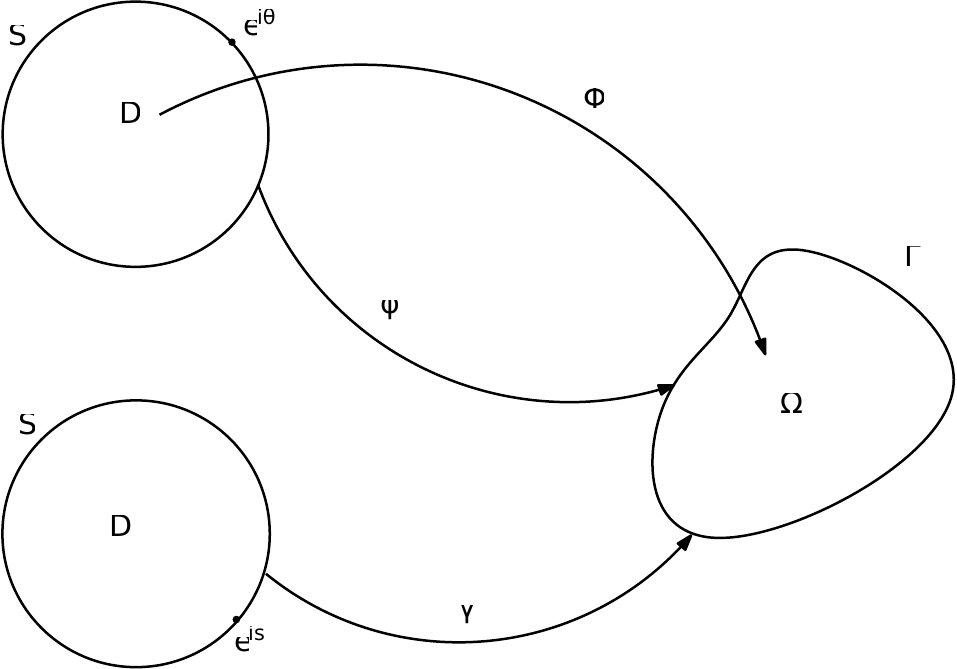}}
\caption{$\psi=\varphi^{-1}\circ\gamma:{\mathbb S}\to{\mathbb S}$.}
    \label{Fig12}
\end{figure}

The diffeomorphism
$\varphi:{\mathbb S}=\{e^{i\theta}\}\to\Gamma$ differs of
$\gamma$ by the change of parameters,
$$
\gamma(e^{is})=\varphi\big(e^{i\Theta(s)}\big),\quad\varphi(e^{i\theta})=\gamma\big(e^{iS(\theta)}\big),
$$
where
$$
S(\theta)=\int\limits_0^\theta|\Phi'(e^{it})|\,dt.
$$
The real function $S(\theta)$ belongs to $C^\infty({\mathbb R})$, satisfies $S'(\theta)>0$ and
$$
S(\theta+2k\pi)=S(\theta)+2k\pi\quad(k\in\Z).
$$
The inverse function $\theta=\Theta(s)$ of $S(\theta)$ has the same properties
\begin{equation}
\Theta'(s)>0,\quad\Theta(s+2k\pi)=\Theta(s)+2k\pi\quad(k\in\Z).
                                \label{5.13}
\end{equation}
Conditions \eqref{5.12} are written in terms of $\Theta$ as follows
\begin{equation}
\Theta(0)=0,\quad\Theta(s_0)=s_0,\quad\Theta(-s_0)=-s_0.
                                \label{5.14}
\end{equation}
The function $\gamma$ is expressed through $\varphi$ and $\Theta(s)$ by
$
\gamma(e^{is})=\varphi(e^{i\Theta(s)})
$
and the diffeomorphism $\psi$ is related to these functions by
\begin{equation}
\psi(e^{is})=e^{i\Theta(s)},\quad \psi^{-1}(e^{i\theta})=e^{iS(\theta)}.
                                \label{5.15}
\end{equation}

Theorem \ref{Th5.1} has the important corollary.

\begin{proposition} \label{P5.2}
Let hypotheses of Theorem \ref{Th5.1} be satisfied and the diffeomorphism $\psi:{\mathbb S}\to{\mathbb S}$ be defined as above.
For every function $u\in C^\infty({\mathbb S})$, there exists a constant $C=C(u)$ such that

\begin{equation}
{\mathcal H}_\gamma u=(\psi^*{\mathcal H}_{\mathbb S}\psi^{*-1})u+C.
                                \label{5.16}
\end{equation}
\end{proposition}

\begin{proof}
Rewrite \eqref{5.16} in the form
\begin{equation}
({\mathcal H}_\gamma\psi^*)u=(\psi^*{\mathcal H}_{\mathbb S})u+C(u).
                                \label{5.17}
\end{equation}
We first prove the validity of \eqref{5.17} for $u\in {\mathcal A}({\mathbb S})$.
Any such function can be represented in the form
$$
u=\widetilde u+C_1,\quad \widetilde u\in{\mathcal A}_0({\mathbb S}),\quad C_1=\mbox{const}.
$$
From this
\begin{equation}
{\mathcal H}_{\mathbb S}u={\mathcal H}_{\mathbb S}\widetilde u
                                \label{5.18}
\end{equation}
and
$
\psi^* u=\psi^*\widetilde u +C_1.
$
Applying the operator ${\mathcal H}_\gamma$ to the latter formula, we have
\begin{equation}
{\mathcal H}_\gamma\psi^* u={\mathcal H}_\gamma\psi^*\widetilde u.
                                \label{5.19}
\end{equation}

By Theorem \ref{Th5.1}, ${\mathcal H}_{\mathbb S}\widetilde u=\widetilde u$. Together with \eqref{5.18}, this gives
$
{\mathcal H}_{\mathbb S}u=\widetilde u.
$
Applying the operator $\psi^*$ to this equality, we have
\begin{equation}
\psi^*{\mathcal H}_{\mathbb S}u=\psi^*\widetilde u.
                                \label{5.20}
\end{equation}

The isomorphism $\psi^*:C^\infty({\mathbb S})\to C^\infty({\mathbb S})$ maps ${\mathcal A}({\mathbb S})$
onto ${\mathcal A}(\gamma)$. Therefore the function $\psi^*\widetilde u\in{\mathcal A}(\gamma)$ can be represented in the form
\begin{equation}
\psi^*\widetilde u=v+C,\quad v\in{\mathcal A}_0(\gamma),\quad C=\mbox{const}.
                                \label{5.21}
\end{equation}
From this
\begin{equation}
{\mathcal H}_\gamma \psi^*\widetilde u={\mathcal H}_\gamma v.
                                \label{5.22}
\end{equation}
Formulas \eqref{5.20} and \eqref{5.21} imply
\begin{equation}
\psi^*{\mathcal H}_{\mathbb S}u=v+C.
                                \label{5.23}
\end{equation}
By Theorem \ref{Th5.1}, ${\mathcal H}_\gamma v=v$. Therefore the formula \eqref{5.22} can be written as
$
{\mathcal H}_\gamma \psi^*\widetilde u=v.
$
This gives together with \eqref{5.19}
\begin{equation}
{\mathcal H}_\gamma \psi^* u=v.
                                \label{5.24}
\end{equation}

Comparing \eqref{5.23} and \eqref{5.24}, we arrive at \eqref{5.17}. We have thus proved the validity of \eqref{5.17} for
$u\in {\mathcal A}({\mathbb S})$.

Quite similarly we prove the validity of \eqref{5.17} for $u\in \overline{{\mathcal A}({\mathbb S})}$.

Being valid for $u\in {\mathcal A}({\mathbb S})$ and for $u\in \overline{{\mathcal A}({\mathbb S})}$, the equality \eqref{5.17}
holds for arbitrary $u\in C^\infty({\mathbb S})$ since
$C^\infty({\mathbb S})={\mathcal A}({\mathbb S})+\overline{{\mathcal A}({\mathbb S})}$ by Theorem \ref{Th5.1}.
\end{proof}

\begin{theorem} \label{Th5.3}
Let hypotheses of Theorem \ref{Th5.1} be satisfied and the function $\Theta(s)$ be defined as above.
For every function $u\in C^\infty({\mathbb S})$, there exists a constant $C=C(u)$ such that
\begin{equation}
({\mathcal H}_\gamma u)(e^{is})
=\frac{1}{2\pi i}\int\limits_0^{2\pi}\Theta'(\sigma)\,\cot\frac{\Theta(\sigma)-\Theta(s)}{2}\,u(e^{i\sigma})\,d\sigma+C.
                                \label{5.25}
\end{equation}
\end{theorem}

\begin{proof}
Replacing $u$ by $\psi^{*-1}u$ in the formula \eqref{5.10}, we obtain
$$
({\mathcal H}_{\mathbb S}\psi^{*-1}u)(e^{i\theta})
=\frac{1}{2\pi i}\int\limits_0^{2\pi}\cot\frac{t-\theta}{2}\,u(\psi^{-1}(e^{it}))\,dt.
$$
By \eqref{5.15}, $\psi^{-1}(e^{it})=e^{iS(t)}$. The latter formula becomes
$$
({\mathcal H}_{\mathbb S}\psi^{*-1}u)(e^{i\theta})
=\frac{1}{2\pi i}\int\limits_0^{2\pi}\cot\frac{t-\theta}{2}\,u(e^{iS(t)})\,dt.
$$
Setting $\theta=\Theta(s)$ here, we obtain
$$
(\psi^*{\mathcal H}_{\mathbb S}\psi^{*-1}u)(e^{is})
=\frac{1}{2\pi i}\int\limits_0^{2\pi}\cot\frac{t-\Theta(s)}{2}\,u(e^{iS(t)})\,dt.
$$
Together with \eqref{5.16}, this gives
$$
({\mathcal H}_\gamma u)(e^{is})=\frac{1}{2\pi i}\int\limits_0^{2\pi}\cot\frac{t-\Theta(s)}{2}\,u(e^{iS(t)})\,dt+C.
$$
Changing the integration variable here by $t=\Theta(\sigma)$, we arrive at \eqref{5.25}.
\end{proof}

It makes sense to mention that $C(u)=0$ in the case when $\Theta(s)=s$ since \eqref{5.25} coincides with \eqref{5.10} in this case.

The constant $C(u)$ can be  eliminated from equation \eqref{5.25}. To this end, we fix some $s_1\in(0,2\pi)$ and set
$s=s_1$ in \eqref{5.25}
$$
({\mathcal H}_\gamma u)(s_1)
=\frac{1}{2\pi i}\int\limits_0^{2\pi}\Theta'(\sigma)\,\cot\frac{\Theta(\sigma)-\Theta(s_0)}{2}\,u(\sigma)\,d\sigma+C(u).
$$
Taking the difference of \eqref{5.25} and the latter equality, we arrive at the equation
\begin{equation}
({\mathcal H}_\gamma u)(s)-({\mathcal H}_\gamma u)(s_1)=
\frac{1}{2\pi i}\int\limits_0^{2\pi}\Theta'(\sigma)
\Big(\cot\frac{\Theta(\sigma)-\Theta(s)}{2}-\cot\frac{\Theta(\sigma)-\Theta(s_1)}{2}\Big)u(\sigma)\,d\sigma.
                                \label{5.26}
\end{equation}
This equation does not contain $C(u)$ but its structure is more complicated than that of \eqref{5.25}.

\section{Numerical solution of the Calder\'{o}n problem} \label{section num}

We shall exploit formula \eqref{5.26} to obtain a reconstruction of the boundary of the domain $\Omega$. For a function
$u\in C^\infty({\mathbb S})$, we write $u(s)$ instead of $u(e^{is})$. Then the formula \eqref{5.26} is written as
\begin{equation}
({\mathcal H}_\gamma u)(\sigma)-({\mathcal H}_\gamma u)(\sigma_1)
=\frac{1}{2\pi i}\int\limits_{-\pi}^{\pi}\Theta'(s)\Big(\cot\frac{\Theta(s)-\Theta(\sigma)}{2}
-\cot\frac{\Theta(s)-\Theta(\sigma_1)}{2}\Big)u(s)\,ds.
                                \label{7.1}
\end{equation}
For $\sigma\neq\sigma_1$, the integral on the right-hand side of \eqref{7.1} is understood in the sense of the principle value at $s=\sigma$
and $s=\sigma_1$.
For convenience, we also reproduce conditions \eqref{5.13}--\eqref{5.14} for $s_0=3\pi/2$
\begin{equation}
\Theta(0)=0,\quad\Theta'(s)>0,\quad\Theta(\pm \frac{3\pi}{2})=\pm\frac{3\pi}{2},\quad\Theta(s+2k\pi)=\Theta(s)+2k\pi\quad(k\in\Z).
                                \label{7.2}
\end{equation}

The Hilbert transform ${\mathcal H}_\gamma$ is given in the inverse problem, i.e., we know the left-hand side of \eqref{7.1} for every function
$u\in C^\infty({\mathbb S})$ and for all $\sigma,\sigma_1\in{\R}$. Our goal is determining the real function $\Theta\in C^\infty({\mathbb R})$
satisfying \eqref{7.2}.
The problem is  overdetermined: We are looking for one function $\Theta(s)$ given the data
$({\mathcal H}_\gamma u)(s)$ for any function
$u\in C^\infty({\mathbb S})$.

\subsection{Differential equations for $\Theta(s)$}
We write \eqref{7.1} as the integral equation
\begin{equation}
2\pi i\Big(({\mathcal H}_\gamma u)(\sigma)-({\mathcal H}_\gamma u)(\sigma_1)\Big)
=\int\limits_{-\pi}^{\pi}K_{\sigma_1}(\sigma,s)\,u(s)\,ds
                                \label{7.3}
\end{equation}
with the kernel
\begin{equation}
K_{\sigma_1}(\sigma,s)=\Theta'(s)\Big(\cot\frac{\Theta(s)-\Theta(\sigma)}{2}
-\cot\frac{\Theta(s)-\Theta(\sigma_1)}{2}\Big).
                                \label{7.4}
\end{equation}

Knowledge of an integral operator is equivalent to knowledge of its kernel. Therefore we can assume the kernel $K_{\sigma_1}(\sigma,s)$ to be
known. We defer the discuss of the numerical procedure of recovering the kernel from the left-hand side of \eqref{7.3} to the next subsection.

Assuming the kernel $K_{\sigma_1}(\sigma,s)$ to be known, we write \eqref{7.4} as an equation in the unknown function $\Theta$
\begin{equation}
\Big(\cot\frac{\Theta(s)-\Theta(\sigma)}{2}
-\cot\frac{\Theta(s)-\Theta(\sigma_1)}{2}\Big)\Theta'(s)=K_{\sigma_1}(\sigma,s).
                                \label{7.5}
\end{equation}
Strictly speaking, \eqref{7.5} is not a differential equation since, besides $\Theta(s)$, it involves $\Theta(\sigma)$ and $\Theta(\sigma_1)$.
Nevertheless, some differential equations can be derived from \eqref{7.5}. Indeed, setting $\sigma=0$ and $\sigma_1=\pm 2\pi/3$ in \eqref{7.5} and
using \eqref{7.2}, we arrive at two differential equations
\begin{equation}
\Big(\cot\frac{\Theta(s)}{2}
-\cot\frac{\Theta(s)-2\pi/3}{2}\Big)\Theta'(s)=K_{2\pi/3}(0,s),
                                \label{7.6}
\end{equation}
\begin{equation}
\Big(\cot\frac{\Theta(s)}{2}
-\cot\frac{\Theta(s)+2\pi/3}{2}\Big)\Theta'(s)=K_{-2\pi/3}(0,s).
                                \label{7.7}
\end{equation}
These equations are of the same structure but they are {\it different}.
One more similar equation is obtained by setting $\sigma=2\pi/3$ and $\sigma_1=-2\pi/3$ in \eqref{7.5}
\begin{equation}
\Big(\cot\frac{\Theta(s)-2\pi/3}{2}-\cot\frac{\Theta(s)+2\pi/3}{2}\Big)\Theta'(s)=K_{-2\pi/3}(2\pi/3,s).
                                \label{7.8}
\end{equation}
The latter equation is just the difference of \eqref{7.6} and \eqref{7.7}. Nevertheless, the equation \eqref{7.8} will be also used.

In what follows, we are going to restrict ourselves by considering equations \eqref{7.6}--\eqref{7.8} only. Before doing the restriction,
we emphasize that, while passing from \eqref{7.5} to \eqref{7.6}--\eqref{7.8}, a lot of information is lost. On the other hand, the passage
simplifies the first part of our calculations: we do not need to compute the function $K_{\sigma_1}(\sigma,s)$ of three variables, we need
to compute three functions $K_{2\pi/3}(0,s)$, $K_{-2\pi/3}(0,s)$ and $K_{-2\pi/3}(2\pi/3,s)$ of one variable.

Each of differential equations \eqref{7.6}--\eqref{7.8} should be accompanied by some initial condition. The coefficient
$$
\cot\frac{\Theta(s)-2\pi/3}{2}-\cot\frac{\Theta(s)+2\pi/3}{2}
$$
of the equation \eqref{7.8} has singularities at $s=-2\pi/3$ and $s=2\pi/3$, but it is a smooth non-vanishing function for
$s\in(-2\pi/3,2\pi/3)$. Therefore we consider the equation \eqref{7.8} on the interval $(-2\pi/3,2\pi/3)$ and choose the initial condition
$\Theta(0)=0$ that is a part of \eqref{7.2}. Thus, for the equation \eqref{7.8}, we have to solve the Cauchy problem
\begin{equation}
\Big(\cot\frac{\Theta(s)-2\pi/3}{2}-\cot\frac{\Theta(s)+2\pi/3}{2}\Big)\Theta'(s)=K^0(s)\ \mbox{for}\ s\in(-2\pi/3,2\pi/3),\quad \Theta(0)=0
                                \label{7.9}
\end{equation}
with the right-hand side
$$
K^0(s)=K_{-2\pi/3}(2\pi/3,s).
$$
Quite similarly, for equations \eqref{7.7} and \eqref{7.6}, we are going to solve Cauchy problems
\begin{equation}
\Big(\cot\frac{\Theta(s)}{2}
-\cot\frac{\Theta(s)+2\pi/3}{2}\Big)\Theta'(s)=K^+(s)\ \mbox{for}\ s\in(0,4\pi/3),\quad \Theta(2\pi/3)=2\pi/3
                                \label{7.10}
\end{equation}
and
\begin{equation}
\Big(\cot\frac{\Theta(s)}{2}
-\cot\frac{\Theta(s)-2\pi/3}{2}\Big)\Theta'(s)=K^-(s)\ \mbox{for}\ s\in(-4\pi/3,0),\quad \Theta(-2\pi/3)=-2\pi/3,
                                \label{7.10a}
\end{equation}
where the right-hand sides are defined by
\begin{equation}
K^+(s)=K_{-2\pi/3}(0,s),\quad K^-(s)=K_{2\pi/3}(0,s).
                                \label{7.11}
\end{equation}
Observe that
\begin{equation}
K^0(s)=K^+(s)-K^-(s).
                                \label{7.11a}
\end{equation}

\subsection{Numerical recovery of  the kernels}
According to \eqref{7.11}--\eqref{7.11a}, we need to compute $K^\pm(s)=K_{\mp2\pi/3}(0,s)$ only.
Setting $\sigma=0$ and $\sigma_1=\mp 2\pi/3$ in \eqref{7.3}, we have
$$
2\pi i\Big(({\mathcal H}_\gamma u)(0)-({\mathcal H}_\gamma u)(\mp 2\pi/3)\Big)
=\int\limits_{-\pi}^{\pi}K^\pm(s)\,u(s)\,ds.
$$
This equality holds for any function $u\in C^\infty({\mathbb S})$. In particular, setting $u(s)=e^{-ins}$, we obtain
\begin{equation}
i\Big(({\mathcal H}_\gamma e^{-ins})(0)-({\mathcal H}_\gamma e^{-ins})(\mp 2\pi/3)\Big)
=\frac{1}{2\pi}\int\limits_{-\pi}^{\pi}K^\pm(s)\,e^{-ins}\,ds.
                                \label{7.13}
\end{equation}

Further arguments of the current subsection are somewhat formal ones. We represent the function $K^\pm$ by its Fourier series
\begin{equation}
K^\pm(s)=\sum\limits_{n\in\Z}\widehat{K^\pm_n}\,e^{ins}.
                                \label{7.14}
\end{equation}
The Fourier coefficients are expressed by
\begin{equation}
\widehat{K^\pm_n}=\frac{1}{2\pi}\int\limits_{-\pi}^{\pi}K^\pm(s)\,e^{-ins}\,ds.
                                \label{7.15}
\end{equation}
Comparing \eqref{7.13} and \eqref{7.15}, we see that
\begin{equation}
\widehat{K^\pm_n}=i\Big(({\mathcal H}_\gamma e^{-ins})(0)-({\mathcal H}_\gamma e^{-ins})(\mp 2\pi/3)\Big).
                                \label{7.16}
\end{equation}
Substitute this expression into \eqref{7.14} to obtain
\begin{equation}
K^\pm(s)=i\sum\limits_{n\in\Z}\Big(({\mathcal H}_\gamma e^{-ins})(0)-({\mathcal H}_\gamma e^{-ins})(\mp 2\pi/3)\Big)e^{ins}.
                                \label{7.17}
\end{equation}

Let us represent the Hilbert transform $\mathcal{H}_{\gamma}$ in the trigonometric basis:
\begin{equation}
{\mathcal H}_\gamma\,e^{ims}=\sum\limits_{n\in\Z}h_{mn}\,e^{ins}\quad(m\in\Z).
                                \label{7.18}
\end{equation}
According to this formula,
\begin{equation}
({\mathcal H}_\gamma e^{-ins})(0)-({\mathcal H}_\gamma e^{-ins})(\mp 2\pi/3)
=\sum\limits_{m\in\Z}h_{-n,m}(1-e^{\mp 2im\pi/3}).
                                \label{7.19}
\end{equation}
We know that $h_{-n,m}$ fast decays as $|m|\to\infty$. Therefore the series on the right-hand side of \eqref{7.19} converges absolutely.
Substituting the expression \eqref{7.19} into \eqref{7.17}, we obtain
\begin{equation}
K^\pm(s)=i\sum\limits_{n\in\Z}\Big(\sum\limits_{m\in\Z}h_{-n,m}(1-e^{\mp 2im\pi/3})\Big)e^{ins}.
                                \label{7.20}
\end{equation}

How about the convergence of the series \eqref{7.20}? Let us consider the simplest example when $\Omega$ coincides with the unit disk
and $\gamma:{\mathbb S}\to{\mathbb S}$ is the identity. As we know, in this case ${\mathcal H}_\gamma e^{ins}=(\mbox{sgn}\,n)\,e^{ins}$, where
$$
\mbox{sgn}\,0=0,\quad\mbox{sgn}\,n=1\ \mbox{for}\ n>0,\quad\mbox{sgn}\,n=-1\ \mbox{for}\ n<0.
$$
Therefore $({\mathcal H}_\gamma e^{-ins})(0)=-\mbox{sgn}\,n$ and $({\mathcal H}_\gamma e^{-ins})(-2\pi/3)=-(\mbox{sgn}\,n)\,e^{2in\pi/3}$.
Substituting these values into \eqref{7.16}, we obtain
$$
\widehat{K^+_n}=-i\,(\mbox{sgn}\,n)(1-e^{2in\pi/3}).
$$
The Fourier coefficients of the kernel $K^+$ are uniformly bounded
$
|\widehat{K^+_n}|\le2
$
but do not decay.

The matrix $(h_{mn})$ of the Hilbert transform is the diagonal matrix in the ``simplest case'': $h_{mn}=(\mbox{sgn}\,m)\delta_{mn}$,
where $\delta_{mn}$ is the Kronecker symbol. The formula \eqref{7.19} becomes
$$
({\mathcal H}_\gamma e^{-ins})(0)-({\mathcal H}_\gamma e^{-ins})(-3\pi/2)=-(\mbox{sgn}\,n)(1-e^{2in\pi/3})
$$
and the formula \eqref{7.20} becomes
$$
K^+(s)=-i\sum\limits_{n\in\Z}(\mbox{sgn}\,n)(1-e^{2in\pi/3})\,e^{ins}.
$$
The series on the right-hand side of the latter formula does not converge. Of course the same is true in the general case: the formula
\eqref{7.20} does not make any sense since the series on the right-hand side does not converge.

The situation can be easily explained. Indeed, we are trying to recover the unknown function $K^+(s)$ by computing its Fourier coefficients.
But the function does not belong to $L^1({\mathbb S})$ since it has two singularities at $s=0$ and $s=4\pi/3$ as is seen from \eqref{7.6}. Therefore the Fourier series \eqref{7.14} does not need to converge, even at the regular points $s\notin\{0,4\pi/3\}$.

Nevertheless, the situation can be easily improved. The idea of the improvement is as follows. We smooth
the function $K^+(s)$ by multiplying it by an appropriately chosen bump function
$\mu^+(s)\in C^\infty({\mathbb S})$ vanishing near singular points of $K^+(s)$. In other words, we are going
to recover {\it the smooth kernel} $L^+(s)=\mu^+(s)K^+(s)$ instead of $K^+(s)$. The function $K^+(s)$ is unknown, but we know that it has
singularities at two points $s=0$ and $s=4\pi/3$. We choose the bump function so that $\mu^+(s)=0$ for $|s|\le\varepsilon/2$ and $\mu^+(s)=0$
for $|s-4\pi/3|\le\varepsilon/2$ for a small $\varepsilon>0$. In virtue of the choice, the product $L^+=\mu^+K^+$ belongs to
$C^\infty({\mathbb S})$ and the Fourier series of $L^+$ converges well. On the other hand, we want $L^+$ to coincide with $K^+$
outside of the singularities. This means that the bump function must satisfy
$$
\mu^+(s)=1\quad\mbox{if}\quad|s|\ge\varepsilon\quad\mbox{and}\quad |s-4\pi/3|\ge\varepsilon.
$$
In particular $L^+(s)=K^+(s)$ for $s\in(\varepsilon,4\pi/3-\varepsilon)$. This is enough for solving the Cauchy problem \eqref{7.10} on the
interval $(\varepsilon,4\pi/3-\varepsilon)$.

The idea will be realized in subsection \ref{subsect C1}. But first we will recall some basic facts of distribution theory.

\subsection{Distributions and Sobolev spaces on the circle}
Let ${\mathcal D}'({\mathbb S})$ be the space of distributions on the unit circle. For a distribution $f\in{\mathcal D}'({\mathbb S})$ and
test function $\varphi\in C^\infty({\mathbb S})$, let $\langle f|\varphi\rangle$ be the value of $f$ on $\varphi$. The embedding
$L^1({\mathbb S})\subset{\mathcal D}'({\mathbb S})$ is defined by
$$
\langle f|\varphi\rangle=\int\limits_0^{2\pi}f(s)\varphi(s)\,ds\quad\big(f\in L^1({\mathbb S}),\varphi\in C^\infty({\mathbb S})\big).
$$
For $f\in{\mathcal D}'({\mathbb S})$ and $\varphi\in C^\infty({\mathbb S})$, the product $\varphi f\in{\mathcal D}'({\mathbb S})$ is defined
by $$
\langle\varphi f|\psi\rangle=\langle f|\varphi\psi\rangle\quad\big(\psi\in C^\infty({\mathbb S})\big).
$$
For a distribution $f\in{\mathcal D}'({\mathbb S})$, Fourier coefficients are defined by
$\widehat f_n=\frac{1}{2\pi}\langle f|e^{-ins}\rangle$.

For a real $r$, the Sobolev space $H^r({\mathbb S})$ consists of distributions $f\in{\mathcal D}'({\mathbb S})$ satisfying
$$
\|f\|^2_{H^r({\mathbb S})}=\sum\limits_{n\in\Z}(1+|n|^{2r})|\widehat f_n|^2<\infty.
$$
Being furnished with this norm, $H^r({\mathbb S})$ is a Hilbert space. One of main statements of distribution theory says that
$$
{\mathcal D}'({\mathbb S})=\bigcup\limits_{r\in\R}H^r({\mathbb S}).
$$
This equality holds for any compact manifold (with a more general definition of Sobolev spaces).
In particular, the equality implies: for $\varphi\in C^\infty({\mathbb S})$ and $f\in{\mathcal D}'({\mathbb S})$, the Fourier coefficients
of the product are expressed by
$$
(\widehat{\varphi f})_n=\sum\limits_{m\in\Z}\widehat f_m\widehat \varphi_{n-m}
$$
and the series on the right-hand side converges absolutely.

\subsection{Smooth kernels} \label{subsect C1}
We introduce a $C^1$ bump function, $\mu^0\in C^1({\mathbb S})$ defined by
\begin{equation}
\mu^0(s)=\left\{\begin{array}{ll}
1&\mbox{for}\ 0\le s\le\pi/3,\\
\frac{1}{2}(1+\cos6s)&\mbox{for}\ \pi/3\le s\le\pi/2,\\
0&\mbox{for}\ \pi/2\le s\le\pi
\end{array}\right.
                                \label{7.25}
\end{equation}

and $\mu^0(-s)=\mu^0(s)$.

The Fourier coefficients of $\mu^0$ can be computed explicitly. We have
\begin{equation}
\begin{aligned}
\widehat{\mu^0_0}&=\frac{5}{12},\quad
\widehat{\mu^0_6}=\widehat{\mu^0_{-6}}=-\frac{1}{24},\\
\widehat{\mu^0_n}&=-\frac{36}{2n(n^2-36)\pi}\Big(\sin\frac{n\pi}{2}+\sin\frac{n\pi}{3}\Big)\quad\mbox{for}\quad n\notin\{0,6,-6\}.
\end{aligned}
                                \label{7.26}
\end{equation}

This implies that $\mu^0\in H^r({\mathbb S})$ for $r<5/2$. The Fourier series
$$
\mu^0(s)=\sum\limits_{n\in\Z}\widehat{\mu^0_n}\,e^{ins}
$$
converges uniformly and absolutely. In practical calculations, we will use the finite sum
$$
\mu^0(s)\approx\sum\limits_{n=-N}^N\widehat{\mu^0_n}\,e^{ins}.
$$

Next, we define two more $C^1$ bump functions $\mu^+,\mu^-\in C^1({\mathbb S})$ by shifting the function $\mu^0$
$$
\mu^+(s)=\mu^0(s-2\pi/3),\quad \mu^-(s)=\mu^0(s+2\pi/3).
$$

The Fourier coefficients of $\mu^+$ and $\mu^-$ are expressed through coefficients \eqref{7.26} by
\begin{equation}
\widehat{\mu^+_n}=e^{-2in\pi/3}\,\widehat{\mu^0_n},\quad\widehat{\mu^-_n}=e^{2in\pi/3}\,\widehat{\mu^0_n}.
                                \label{7.28}
\end{equation}

As $\mathrm{supp}(\mu^{0}) = [-\pi/2,\pi/2]$ and $K^{0}$ is $C^{\infty}$ in $[-\pi/2-\varepsilon,\pi/2+\varepsilon]$ for $\varepsilon$ small
enough, we have that $L^{0} \in H^r({\mathbb S})$ for $r < \frac{5}{2}$. Similarly $L^{\pm} \in H^r({\mathbb S})$ for $r < \frac{5}{2}$.
The smooth kernels $L^0,L^\pm\in C^1({\mathbb S})$ are defined by
\begin{equation}
L^0(s)=\mu^0(s)K^0(s),\quad L^+(s)=\mu^+(s)K^+(s),\quad L^-(s)=\mu^-(s)K^-(s).
                                \label{7.29}
\end{equation}
Their Fourier coefficients are expressed by
\begin{equation}
\widehat{L^0_n}=  \sum\limits_{m\in\Z}\widehat{K^0_m}\widehat{\mu^0_{n-m}},\quad
\widehat{L^\pm_n}=\sum\limits_{m\in\Z}\widehat{K^\pm_m}\widehat{\mu^\pm_{n-m}}
                                \label{7.30}
\end{equation}
and these series converge absolutely.
Assuming $\widehat{K^\pm_n}$ and $\widehat{K^0_n}$ have been computed by \eqref{7.16} and \eqref{7.11a}, we can compute
$\widehat{L^0_n},\widehat{L^\pm_n}$ and recover the functions $L^0,L^\pm$ by
\begin{equation}
L^0(s)=\sum\limits_{n\in\Z}\widehat{L^0_n}\,e^{ins},\quad
L^\pm(s)=\sum\limits_{n\in\Z}\widehat{L^\pm_n}\,e^{ins}.
                                \label{7.31}
\end{equation}
Unlike the divergent series \eqref{7.14}, the Fourier series \eqref{7.31} converges uniformly and absolutely as $L^{0}(s)$ and $L^{\pm}(s)$ are
in $H^r({\mathbb S})$ for $r<\frac{5}{2}$. Thus, we know the functions $L^0$ and $L^\pm$.

As is seen from \eqref{7.25} and \eqref{7.29},
$$
\begin{aligned}
K^0(s)&=L^0(s)\quad \mbox{for}\quad s\in[-\pi/3,\pi/3],\\
K^+(s)&=L^+(s)\quad \mbox{for}\quad s\in[\pi/3,\pi],\\
K^-(s)&=L^-(s)\quad \mbox{for}\quad s\in[-\pi,-\pi/3].
\end{aligned}
$$
Therefore formulas \eqref{7.9} and \eqref{7.10}--\eqref{7.10a} imply
\begin{equation}
\Big(\cot\frac{\Theta(s)-2\pi/3}{2}-\cot\frac{\Theta(s)+2\pi/3}{2}\Big)\Theta'(s)=L^0(s)\ \mbox{for}\ s\in[-\pi/3,\pi/3],\quad \Theta(0)=0;
                                \label{7.32}
\end{equation}
\begin{equation}
\Big(\cot\frac{\Theta(s)}{2}
-\cot\frac{\Theta(s)+2\pi/3}{2}\Big)\Theta'(s)=L^+(s)\quad \mbox{for}\ s\in[\pi/3,\pi],\quad \Theta(2\pi/3)=2\pi/3;
                                \label{7.33}
\end{equation}
\begin{equation}
\Big(\cot\frac{\Theta(s)}{2}
-\cot\frac{\Theta(s)-2\pi/3}{2}\Big)\Theta'(s)=L^-(s)\quad \mbox{for}\ s\in[-\pi,-\pi/3],\quad \Theta(-2\pi/3)=-2\pi/3.
                                \label{7.34}
\end{equation}
Right-hand sides of these equations belong to $C^1({\mathbb S})$ and are given by their Fourier series. We can solve Cauchy problems
\eqref{7.32}--\eqref{7.34} and find the function $\Theta(s)$ on the segments $[-\pi/3,\pi/3],[\pi/3,\pi]$ and $[-\pi,-\pi/3]$.
By virtue of $\Theta(s+2\pi)=\Theta(s)+2\pi$ (see \eqref{7.32}), we can also find the function $\Theta(s)$ on the segment $[\pi,2\pi-\pi/3]$.
Thus, the function $\Theta(s)$ is known on $[-\pi/3,2\pi-\pi/3]$. Together with $\Theta(s+2k\pi)=\Theta(s)+2k\pi\ (k\in\Z)$, this determines
$\Theta(s)$ for all real $s$.

Let us discuss these circumstances in more details.
Let us denote solutions to the Cauchy problems \eqref{7.32}--\eqref{7.34} by $\Theta^0$ and $\Theta^\pm$ respectively. Then the following
equalities must hold:
\begin{equation}
\Theta^0(\pi/3)=\Theta^+(\pi/3),\quad
\Theta^+(\pi)=\Theta^-(-\pi)+2\pi,\quad
\Theta^-(-\pi/3)=\Theta^0(-\pi/3).
                                \label{7.35}
\end{equation}
These equalities provide control for the reconstruction algorithm.

\subsection{Cauchy problems (4.33)--(4.35) can be solved analytically}
Let us discuss the Cauchy problem \eqref{7.32}. The left-hand side of the equation \eqref{7.32} can be represented as the derivative
$$
\Big(\cot\frac{\Theta(s)-2\pi/3}{2}-\cot\frac{\Theta(s)+2\pi/3}{2}\Big)\Theta'(s)
=2\frac{d}{ds}\left(\ln\frac{\Big|\sin\frac{\Theta(s)-2\pi/3}{2}\Big|}{\left|\sin\frac{\Theta(s)+2\pi/3}{2}\right|}\right).
$$
The equation \eqref{7.32} is now written as
$$
\frac{d}{ds}\left(\ln\frac{\Big|\sin\frac{\Theta(s)-2\pi/3}{2}\Big|}{\left|\sin\frac{\Theta(s)+2\pi/3}{2}\right|}\right)=
\frac{1}{2}L^0(s).
$$
Using the initial condition $\Theta(0)=0$, we obtain
\begin{equation}
\frac{\Big|\sin\frac{\Theta(s)-2\pi/3}{2}\Big|}{\left|\sin\frac{\Theta(s)+2\pi/3}{2}\right|}
=\exp\Big(\frac{1}{2}\int\limits_0^s L^0(\sigma)\,d\sigma\Big)\quad(s\in[-\pi/3,\pi/3]).
                                \label{7.36}
\end{equation}
We try to express $\Theta(s)$ from \eqref{7.36}. We first determine signs of $\sin\frac{\Theta(s)-2\pi/3}{2}$ and of
$\sin\frac{\Theta(s)+2\pi/3}{2}$.
Since $\Theta'(s)>0$  and $\Theta(-2\pi/3)=-2\pi/3, \Theta(2\pi/3)=2\pi/3$, we have
$$
-2\pi/3< \Theta(s)<2\pi/3\quad\mbox{for}\quad s\in[-\pi/3,\pi/3].
$$
From this
$$
-2\pi/3<\frac{\Theta(s)-2\pi/3}{2}<0,\quad 0<\frac{\Theta(s)+2\pi/3}{2}<2\pi/3\quad\mbox{for}\quad s\in[-\pi/3,\pi/3].
$$
Thus, $\sin\frac{\Theta(s)-2\pi/3}{2}<0$ and $\sin\frac{\Theta(s)+2\pi/3}{2}>0$ for $s\in[-\pi/3,\pi/3]$. The equation \eqref{7.36} simplifies
to the following one:
\begin{equation}
\frac{\sin\frac{\Theta(s)-2\pi/3}{2}}{\sin\frac{\Theta(s)+2\pi/3}{2}}
=-\exp\Big(\frac{1}{2}\int\limits_0^s L^0(\sigma)\,d\sigma\Big)\quad(s\in[-\pi/3,\pi/3]).
                                \label{7.37}
\end{equation}
Next, we use the identity
$$
\begin{aligned}
\sin\frac{\Theta(s)-2\pi/3}{2}&=\sin\Big(\frac{\Theta(s)+2\pi/3}{2}-\frac{2\pi}{3}\Big)\\
&=\cos\frac{2\pi}{3}\sin\frac{\Theta(s)+2\pi/3}{2}-\sin\frac{2\pi}{3}\cos\frac{\Theta(s)+2\pi/3}{2}\\
&=-\frac{1}{2}\sin\frac{\Theta(s)+2\pi/3}{2}-\frac{\sqrt{3}}{2}\cos\frac{\Theta(s)+2\pi/3}{2}.
\end{aligned}
$$
Replacing the nominator on the left-hand side of \eqref{7.37} with the latter expression, we obtain
$$
-\frac{1}{2}-\frac{\sqrt{3}}{2}\cot\frac{\Theta(s)+2\pi/3}{2}
=-\exp\Big(\frac{1}{2}\int\limits_0^s L^0(\sigma)\,d\sigma\Big).
$$
From this we express
$$
\cot\frac{\Theta(s)+2\pi/3}{2}
=\frac{1}{\sqrt{3}}\bigg[2\exp\Big(\frac{1}{2}\int\limits_0^s L^0(\sigma)\,d\sigma\Big)-1\bigg]
$$
and
\begin{equation}
\tan\frac{\Theta(s)+2\pi/3}{2}
=\sqrt{3}\bigg[2\exp\Big(\frac{1}{2}\int\limits_0^s L^0(\sigma)\,d\sigma\Big)-1\bigg]^{-1}.
                                \label{7.37a}
\end{equation}
Then we use the obvious relations
$$
\tan\frac{\Theta(s)}{2}=\tan\Big(\frac{\Theta(s)+2\pi/3}{2}-\frac{\pi}{3}\Big)
=\frac{\tan\frac{\Theta(s)+2\pi/3}{2}-\tan\frac{\pi}{3}}{1+\tan\frac{\pi}{3}\tan\frac{\Theta(s)+2\pi/3}{2}}
=\frac{\tan\frac{\Theta(s)+2\pi/3}{2}-\sqrt{3}}{1+\sqrt{3}\tan\frac{\Theta(s)+2\pi/3}{2}}.
$$
This gives together with \eqref{7.37a}
$$
\tan\frac{\Theta(s)}{2}
=\frac{\sqrt{3}\bigg[2\exp\Big(\frac{1}{2}\int\limits_0^s L^0(\sigma)\,d\sigma\Big)-1\bigg]^{-1}-\sqrt{3}}
{1+3\bigg[2\exp\Big(\frac{1}{2}\int\limits_0^s L^0(\sigma)\,d\sigma\Big)-1\bigg]^{-1}}.
$$
Multiply the nominator and denominator on the right-hand side by $2\exp\Big(\frac{1}{2}\int\limits_0^s L^0(\sigma)\,d\sigma\Big)-1$ to obtain
$$
\tan\frac{\Theta(s)}{2}
=\sqrt{3}\,\,\frac{1-\exp\Big(\frac{1}{2}\int\limits_0^s L^0(\sigma)\,d\sigma\Big)}
{1+\exp\Big(\frac{1}{2}\int\limits_0^s L^0(\sigma)\,d\sigma\Big)}.
$$
We thus arrive at the final formula
\begin{equation}
\Theta(s)=2\arctan\left[\sqrt{3}\,\,\frac{1-\exp\Big(\frac{1}{2}\int\limits_0^s L^0(\sigma)\,d\sigma\Big)}
{1+\exp\Big(\frac{1}{2}\int\limits_0^s L^0(\sigma)\,d\sigma\Big)}\right]
\qquad(s\in[-\pi/3,\pi/3]).
                                \label{7.38}
\end{equation}

In the same way we find the solutions to Cauchy problems \eqref{7.33} and \eqref{7.34}
\begin{equation}
\Theta(s)=2\arctan\left[\sqrt{3}\,\,\frac{1-\exp\Big(-\frac{1}{2}\int\limits_{2\pi/3}^s L^+(\sigma)\,d\sigma\Big)}
{1+\exp\Big(-\frac{1}{2}\int\limits_{2\pi/3}^s L^+(\sigma)\,d\sigma\Big)}\right]+2\pi/3
\qquad(s\in[\pi/3,\pi]),
                                \label{7.43}
\end{equation}
\begin{equation}
\Theta(s)=-2\arctan\left[\sqrt{3}\,\,\frac{1-\exp\Big(-\frac{1}{2}\int\limits_{-2\pi/3}^s L^-(\sigma)\,d\sigma\Big)}
{1+\exp\Big(-\frac{1}{2}\int\limits_{-2\pi/3}^s L^-(\sigma)\,d\sigma\Big)}\right]-2\pi/3
\qquad(s\in[-\pi,-\pi/3]).
                                \label{7.45}
\end{equation}
The integrals on r.h.s. of \eqref{7.38}--\eqref{7.45} can be also expressed by an explicit formula. Indeed, as follows from \eqref{7.31},
\begin{equation}
\int\limits_0^s L^0(\sigma)\,d\sigma=
\widehat{L^0_0}\,s-i\sum\limits_{n\in\Z\setminus\{0\}}\frac{1}{n}\widehat{L^0_n}(e^{ins}-1),
                                \label{7.39}
\end{equation}
\begin{equation}
\int\limits_{2\pi/3}^s L^+(\sigma)\,d\sigma=
\widehat{L^+_0}(s-2\pi/3)-i\sum\limits_{n\in\Z\setminus\{0\}}\frac{1}{n}\widehat{L^+_n}(e^{ins}-e^{2in\pi/3}),
                                \label{7.44}
\end{equation}
\begin{equation}
\int\limits_{-2\pi/3}^s L^-(\sigma)\,d\sigma=
\widehat{L^-_0}(s+2\pi/3)-i\sum\limits_{n\in\Z\setminus\{0\}}\frac{1}{n}\widehat{L^-_n}(e^{ins}-e^{-2in\pi/3}).
                                \label{7.46}
\end{equation}

\subsection{Recovering the domain $\Omega$}
After the function $\Theta(s)$ has been determined, the domain $\Omega$ can be easily recovered as follows.

Given the function $\theta=\Theta(s)$, the inverse function $s=S(\theta)$ is well defined since $\Theta'>0$. By \eqref{7.2}, the real
function $S\in C^\infty({\mathbb R})$ possesses similar properties
\begin{equation}
S(0)=0,\quad S'(\theta)>0,\quad S(\pm \frac{3\pi}{2})=\pm\frac{3\pi}{2},\quad S(\theta+2k\pi)=S(\theta)+2k\pi\quad(k\in\Z).
                                \label{7.47}
\end{equation}
Then we introduce the function $0<a\in C^\infty({\mathbb S})$ by
\begin{equation}
a(\theta)=\frac{1}{S'(\theta)}.
                                \label{7.48}
\end{equation}
It satisfies the normalization condition
\begin{equation}
\frac{1}{2\pi}\int\limits_{-\pi}^\pi\frac{d\theta}{a(\theta)}=1.
                                \label{7.48a}
\end{equation}
The condition plays an important role in \cite{JS0,JS2}, but we have not used it before.

Given the function $a$, the domain $\Omega$ is recovered by the following procedure taken from the proof of \cite[Lemma 1]{E}.

Given a function $0<a\in C^\infty({\mathbb S})$, we first set $h(\theta)=-\ln a(\theta)$. Let
\begin{equation}
h(\theta)=\sum\limits_{n=-\infty}^\infty{\widehat h}_n e^{in\theta}
                               \label{7.49}
\end{equation}
be the Fourier series of $h$. Define $k(\theta)$ by
\begin{equation}
k(\theta)={\widehat h}_0+2\sum\limits_{n=1}^\infty{\widehat h}_n e^{in\theta}.
                               \label{7.50}
\end{equation}
Then $e^k$ extends to a non-vanishing holomorphic function on the disc ${\mathbb D}$, whose absolute value at the boundary is
$e^{\RealPart(k)}=a^{-1}$.
Hence we can recover the biholomorphism $\Phi$ from \eqref{5.11} as
\begin{equation}
\Phi(z)=\int\limits_{z_0}^z e^{k(\zeta)}\,d\zeta\quad(z\in{\mathbb D}).
                               \label{7.51}
\end{equation}
The domain $\Omega=\Phi({\mathbb D})$ is determined by the function $a$ uniquely up to a shift and rotation.

\subsection{Computing the function $a(\theta)$}
We represent the function $a$ by its Fourier series
\begin{equation}
a(\theta)=\sum\limits_{n=-\infty}^\infty{\widehat a}_n e^{in\theta}.
                               \label{7.52}
\end{equation}
By \eqref{7.48},
$$
{\widehat a}_n=\frac{1}{2\pi}\int\limits_{-\pi}^\pi e^{-in\theta} a(\theta)\,d\theta
=\frac{1}{2\pi}\int\limits_{-\pi}^\pi e^{-in\theta} \,\frac{d\theta}{S'(\theta)}.
$$
Changing the integration variable by
$$
\theta=\Theta(s),\quad s=S(\theta),\quad d\theta=\Theta'(s)\,ds,\quad \Theta(s)S'(\theta)=\frac{1}{\Theta'(s)},
$$
we obtain
$$
{\widehat a}_n=\frac{1}{2\pi}\int\limits_{-\pi}^\pi e^{-in\Theta(s)} \big(\Theta'(s)\big)^2\,ds.
$$
Integration limits have not been changed since the integrand is a $2\pi$-periodic function.
We split the latter integral to three parts
\begin{equation}
{\widehat a}_n=\frac{1}{2\pi}\int\limits_{-\pi/3}^{\pi/3} e^{-in\Theta(s)} \big(\Theta'(s)\big)^2\,ds
+\frac{1}{2\pi}\int\limits_{\pi/3}^{\pi} e^{-in\Theta(s)} \big(\Theta'(s)\big)^2\,ds
+\frac{1}{2\pi}\int\limits_{-\pi}^{-\pi/3} e^{-in\Theta(s)} \big(\Theta'(s)\big)^2\,ds.
                               \label{7.53}
\end{equation}

Now, we compute the first integral on the right-hand side of \eqref{7.53}.
Introducing the function
\begin{equation}
E^0(s)=\exp\Big(\frac{1}{2}\int\limits_0^s L^0(\sigma)\,d\sigma\Big)\quad(s\in[-\pi/3,\pi/3]),
                               \label{7.54}
\end{equation}
we rewrite \eqref{7.38} in the form
$$
\Theta(s)=2\arctan\Big(\sqrt{3}\,\,\frac{1-E^0(s)}{1+E^0(s)}\Big)\quad(s\in[-\pi/3,\pi/3]).
$$
From this
$$
e^{-in\Theta(s)}=\exp\Big[-2in\arctan\Big(\sqrt{3}\,\,\frac{1-E^0(s)}{1+E^0(s)}\Big)\Big]
$$
and
$$
(\Theta'(s))^2=\frac{3}{4}\,\frac{(L^0(s))^2(E^0(s))^2}{\Big(1-E^0(s)+(E^0(s))^2\Big)^2}.
$$
Two last formulas imply
\begin{equation}
\begin{aligned}
\frac{1}{2\pi}&\int\limits_{-\pi/3}^{\pi/3} e^{-in\Theta(s)} \big(\Theta'(s)\big)^2\,ds=\\
&=\frac{3}{8\pi}\int\limits_{-\pi/3}^{\pi/3}
\exp\Big[-2in\arctan\Big(\sqrt{3}\,\,\frac{1-E^0(s)}{1+E^0(s)}\Big)\Big]
\Big(\frac{L^0(s)E^0(s)}{1-E^0(s)+(E^0(s))^2}\Big)^2\,ds.
\end{aligned}
                               \label{7.58}
\end{equation}

Two last integrals on the right-hand side of \eqref{7.53} are expressed in the same way:
\begin{equation}
E^+(s)=\exp\Big(-\frac{1}{2}\int\limits_{2\pi/3}^s L^+(\sigma)\,d\sigma\Big)\quad(s\in[\pi/3,\pi]),
                               \label{7.59}
\end{equation}
\begin{equation}
E^-(s)=\exp\Big(-\frac{1}{2}\int\limits_{-2\pi/3}^s L^-(\sigma)\,d\sigma\Big)\quad(s\in[-\pi,-\pi/3]);
                               \label{7.60}
\end{equation}
\begin{equation}
\begin{aligned}
\frac{1}{2\pi}&\int\limits_{\pi/3}^{\pi} e^{-in\Theta(s)} \big(\Theta'(s)\big)^2\,ds=\\
&=\frac{3}{8\pi}e^{-2in\pi/3}\int\limits_{\pi/3}^{\pi}
\exp\Big[-2in\arctan\Big(\sqrt{3}\,\,\frac{1-E^+(s)}{1+E^+(s)}\Big)\Big]
\Big(\frac{L^+(s)E^+(s)}{1-E^+(s)+(E^+(s))^2}\Big)^2\,ds,
\end{aligned}
                               \label{7.61}
\end{equation}
\begin{equation}
\begin{aligned}
\frac{1}{2\pi}&\int\limits_{-\pi}^{-\pi/3} e^{-in\Theta(s)} \big(\Theta'(s)\big)^2\,ds=\\
&=\frac{3}{8\pi}e^{2in\pi/3}\int\limits_{-\pi}^{-\pi/3}
\exp\Big[2in\arctan\Big(\sqrt{3}\,\,\frac{1-E^-(s)}{1+E^-(s)}\Big)\Big]
\Big(\frac{L^-(s)E^-(s)}{1-E^-(s)+(E^-(s))^2}\Big)^2\,ds.
\end{aligned}
                               \label{7.62}
\end{equation}
Substituting expressions \eqref{7.58}, \eqref{7.61}--\eqref{7.62} into \eqref{7.53}, we arrive at the final formula
\begin{equation}
\begin{aligned}
{\widehat a}_n&=
\frac{3}{8\pi}\int\limits_{-\pi/3}^{\pi/3}
\exp\Big[-2in\arctan\Big(\sqrt{3}\,\,\frac{1-E^0(s)}{1+E^0(s)}\Big)\Big]
\Big(\frac{L^0(s)E^0(s)}{1-E^0(s)+(E^0(s))^2}\Big)^2\,ds\\
&+\frac{3}{8\pi}e^{-2in\pi/3}\int\limits_{\pi/3}^{\pi}
\exp\Big[-2in\arctan\Big(\sqrt{3}\,\,\frac{1-E^+(s)}{1+E^+(s)}\Big)\Big]
\Big(\frac{L^+(s)E^+(s)}{1-E^+(s)+(E^+(s))^2}\Big)^2\,ds\\
&+\frac{3}{8\pi}e^{2in\pi/3}\int\limits_{-\pi}^{-\pi/3}
\exp\Big[2in\arctan\Big(\sqrt{3}\,\,\frac{1-E^-(s)}{1+E^-(s)}\Big)\Big]
\Big(\frac{L^-(s)E^-(s)}{1-E^-(s)+(E^-(s))^2}\Big)^2\,ds.
\end{aligned}
                               \label{7.63}
\end{equation}
The formula is pretty bulky, but it involves neither $\Theta(s)$ nor $S(\theta)$. And what is the most important advantage of the formula
\eqref{7.63}, it allows us to avoid computation of the derivative $S'(\theta)$ participating in \eqref{7.48}.

Thus, given the functions $L^0(s)$ and $L^\pm(s)$, we compute $E^0(s)$ and $E^\pm(s)$ by formulas \eqref{7.54} and \eqref{7.59}--\eqref{7.60},
and then compute the Fourier coefficients ${\widehat a}_n$ by \eqref{7.63}. Recall also that integrals
$\int_0^sL^0\,d\sigma,\int_{2\pi/3}^sL^+\,d\sigma$ and $\int_{-2\pi/3}^sL^-\,d\sigma$ participating in \eqref{7.54} and
\eqref{7.59}--\eqref{7.60} can be computed by explicit formulas \eqref{7.39}, \eqref{7.44} and \eqref{7.46}.

After the Fourier coefficients ${\widehat a}_n$ have been computed, we compute the function $a(\theta)$ itself by \eqref{7.52}.

\subsection{The algorithm}\label{Reconalgo}
Concisely, the numerical reconstruction of a simply connected, probably multi-sheeted, domain $\Omega$ from the Hilbert transform
${\mathcal H}_\gamma$ consists of the following steps.

{\bf Preliminary step.} We compute Fourier coefficients of the bump functions $\mu^0,\mu^\pm$ by formulas \eqref{7.26}, \eqref{7.28}. The data
${\mathcal H}_\gamma$ is not used in the preliminary step.

Given the truncated matrix $(h_{mn})_{m,n=-N}^N$ of the operator ${\mathcal H}_\gamma$ (see \eqref{7.18}), we do the following.

{\bf Step 1.} Compute the Fourier coefficients $\widehat{K^\pm_n}$ of the kernels $K^\pm$ by the formula \eqref{7.16}. Then compute the Fourier
coefficients $\widehat{K^0_n}$ of $K^0$ by \eqref{7.11a}. We  emphasize that the kernels $K^0(s),K^\pm(s)$ themselves remain unknown. We do not
compute them since the Fourier series on the right-hand side of \eqref{7.14} is a divergent series.

{\bf Step 2.} Compute the Fourier coefficients $\widehat{L^0_n},\widehat{L^\pm_n}$ of smooth kernels by \eqref{7.30}. Then compute the functions
$L^0(s),L^\pm(s)$ by \eqref{7.31}; Fourier series \eqref{7.31} converge uniformly and absolutely since the functions $L^0(s),L^\pm(s)$ are
$C^1$-smooth. Then compute the functions $E^0(s),E^\pm(s)$ by formulas \eqref{7.54} and \eqref{7.59}--\eqref{7.60}. We recall that
$\int_0^sL^0\,d\sigma,\int_{2\pi/3}^sL^+\,d\sigma$ and $\int_{-2\pi/3}^sL^-\,d\sigma$  in \eqref{7.54}, \eqref{7.59} and \eqref{7.60} can be
computed by explicit formulas \eqref{7.39}--\eqref{7.46}.

{\bf Step 3.} Compute the Fourier coefficients ${\widehat a}_n$ of the function $a$ by \eqref{7.63}. Then compute the function $a(\theta)$
itself by \eqref{7.52}. Plot the graph of $a(\theta)$. Normalize  the function $a(\theta)$ by multiplying it by a positive constant chosen
so that the condition \eqref{7.48a} holds exactly.

{\bf Step 4.} Given the function $a(\theta)$, recover the holomorphic function $\Phi(z)$ on the unit disk $\D$ by formulas
\eqref{7.47}--\eqref{7.51}.
Then $\Omega=\Phi({\mathbb D})$. Finally, the required curve is $\Gamma=\{\Phi(e^{i\theta})\mid0\le\theta\le2\pi\}$.

\bigskip

\section{Numerical reconstructions} \label{numerical reconstructios}

\subsection{The forward problem}
In practical tomography, the truncated matrix $(\lambda_{mn})_{m,n=-N}^N$ of the DN operator is obtained from results of voltage -- current measurements.
In our numerical experiments, we compute the matrix for a given compact domain $\Omega\subset{\R}^2$ bounded by a smooth curve $\Gamma$. This first part
of the experiment is referred to as {\it the forward problem}. We first shortly discuss the numerical solution of  the forward problem.

Our algorithm for solving the inverse problem uses the matrix $(h_{mn})_{m,n=-N}^N$ of the Hilbert transform ${\mathcal H}_\gamma$.
As is seen from \eqref{5.0}--\eqref{5.1} and \eqref{7.17}
$$
h_{mn}=\frac{1}{m}\lambda_{mn}\ (m\neq0),\quad h_{0,n}=0.
$$

By the definition of the DN map,
\begin{equation}
\Lambda_\Gamma f=\frac{\partial u}{\partial\nu}\Big|_\Gamma,
                               \label{8.1}
\end{equation}
where $u$ is the solution to the Dirichlet problem
\begin{equation}
\Delta u = 0 \  \text{in } \Omega, \quad
u = f \ \text{on } \partial \Omega.
                               \label{8.2}
\end{equation}
As is well known, the Diriclet problem \eqref{8.2} is equivalent to a second kind integral equation on the curve $\Gamma$.
We follow the procedure in \cite{Nas} and \cite{Naq} of solving the integral equation.

Numerical implementation of the differentiation in \eqref{8.1} leads to some instability. To avoid the numerical differentiation, we use the following
well known trick. Let $v$ be the harmonic conjugate of the solution $u$ to the Diriclet problem \eqref{8.2}. At least for a simply connected $\Omega$,
the function $v$ is obtained from $u$ by integration that does not lead to instability. Then
$$
\frac{\partial u}{\partial\nu}(\gamma(s))=\frac{d(v(\gamma(s))}{ds},
$$
where $\gamma(s)$ is the arc-length parametrization of $\Gamma$, see
\cite[Theorem 3.1]{Naq}. Unlike $\frac{\partial u}{\partial\nu}$, the derivative $\frac{d(v(\gamma(s))}{ds}$ is expressed analytically and
does not need numerical differentiation if the Fourier series of $v(\gamma(s))$ is known.

For solving the forward problem for a simply connected
multi-sheeted domain $\Omega$, we use the following fact. If $\Phi : {\mathbb D}\rightarrow \Omega$ is a biholomorphism (${\mathbb D}$ is the
unit disk) and $\varphi=\Phi|_{\mathbb S} : {\mathbb S}\rightarrow \Gamma$, then
$\Lambda_{\mathbb D}=|\Phi'|_{\mathbb S}|\varphi^*\Lambda_\Omega\varphi^{*-1}$, see \cite{JS3}.

\begin{figure}[H]
    \centering

    \begin{subfigure}[t]{0.3\textwidth}
        \centering
        \includegraphics[width=\linewidth]{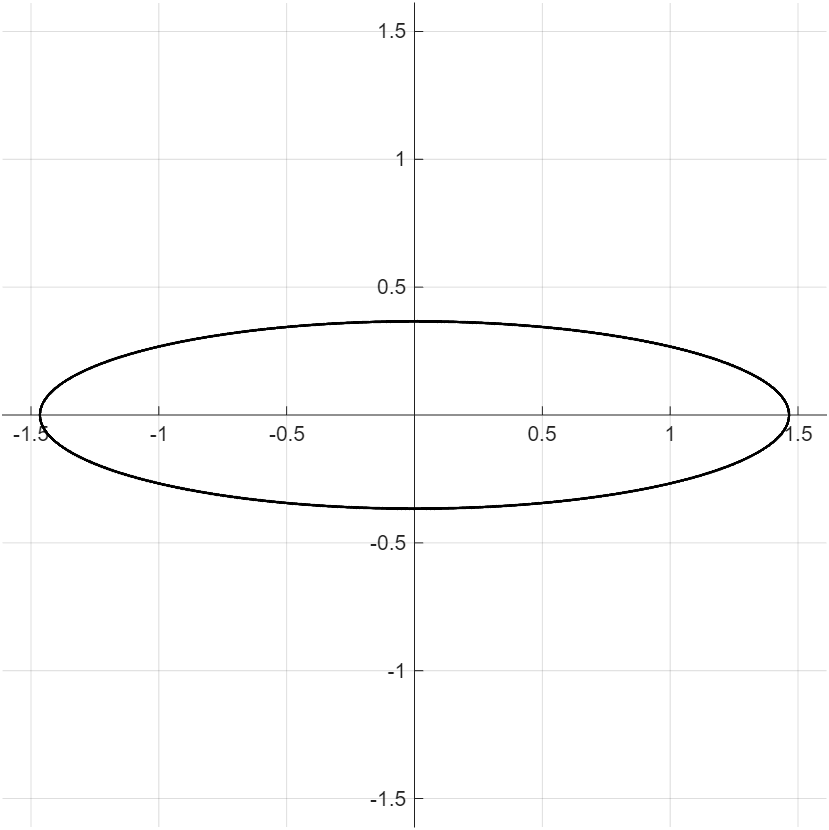}
        \caption{Original domain}
        \label{ellipse_recon_a}
    \end{subfigure}
    \hspace{0.1\textwidth}
    \begin{subfigure}[t]{0.3\textwidth}
        \centering
        \includegraphics[width=\linewidth]{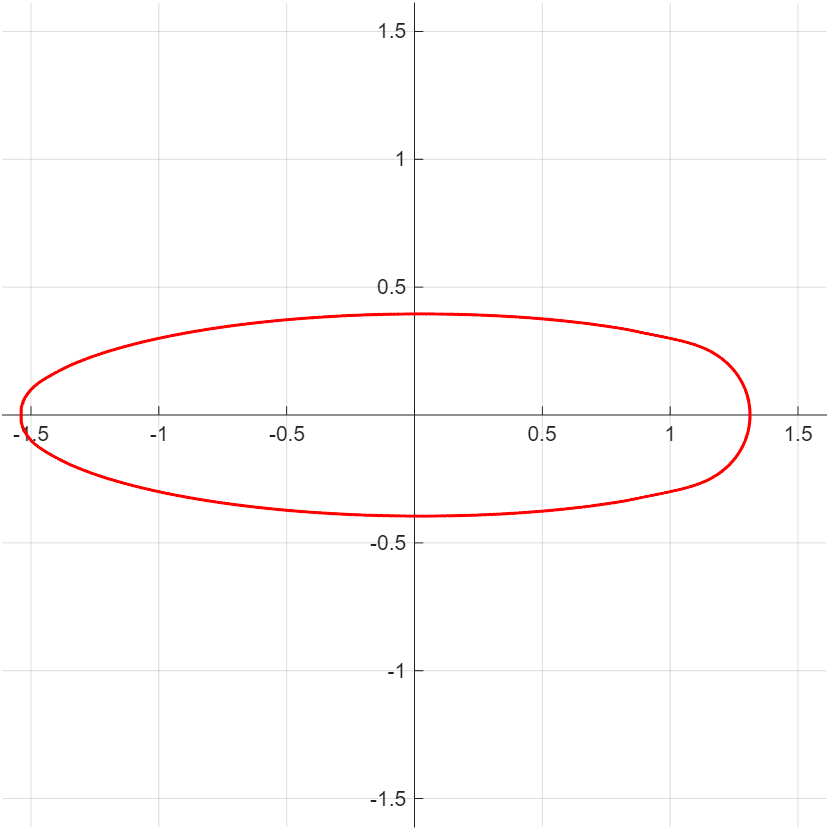}
        \caption{Recovered domain using $\pm 500$ modes of $a$ and $\pm 1000$ modes of $h$}
    \end{subfigure}
    \\[0.5cm]
    \begin{subfigure}[t]{0.3\textwidth}
        \centering
        \includegraphics[width=\linewidth]{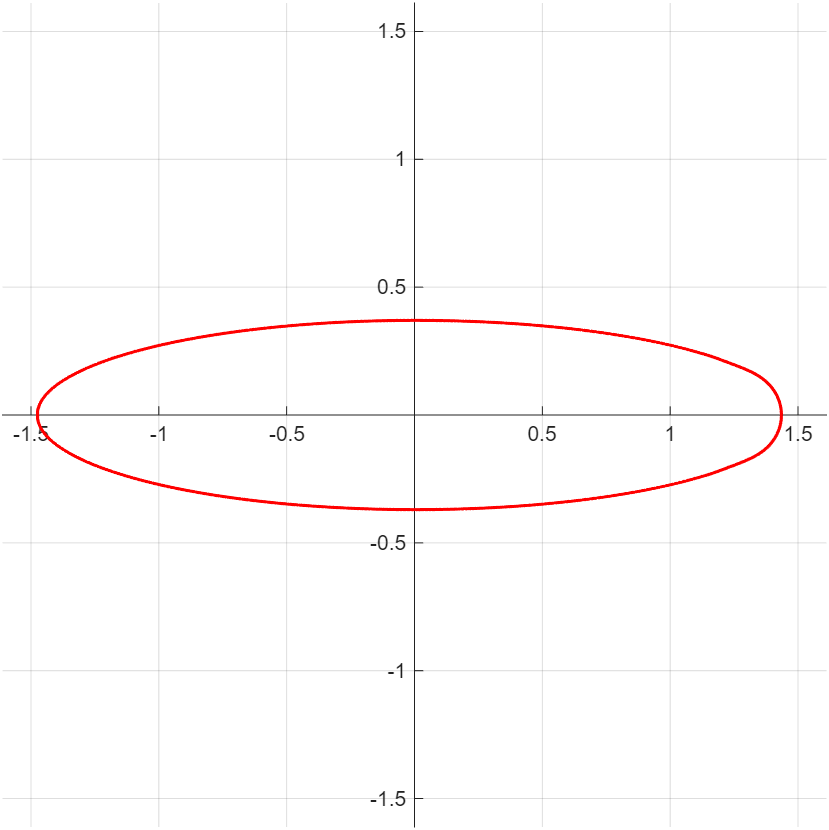}
        \caption{Recovered domain using $\pm 5000$ modes of $a$ and $\pm 10000$ modes of $h$}
    \end{subfigure}
    \hspace{0.1\textwidth}
    \begin{subfigure}[t]{0.3\textwidth}
        \centering
        \includegraphics[width=\linewidth]{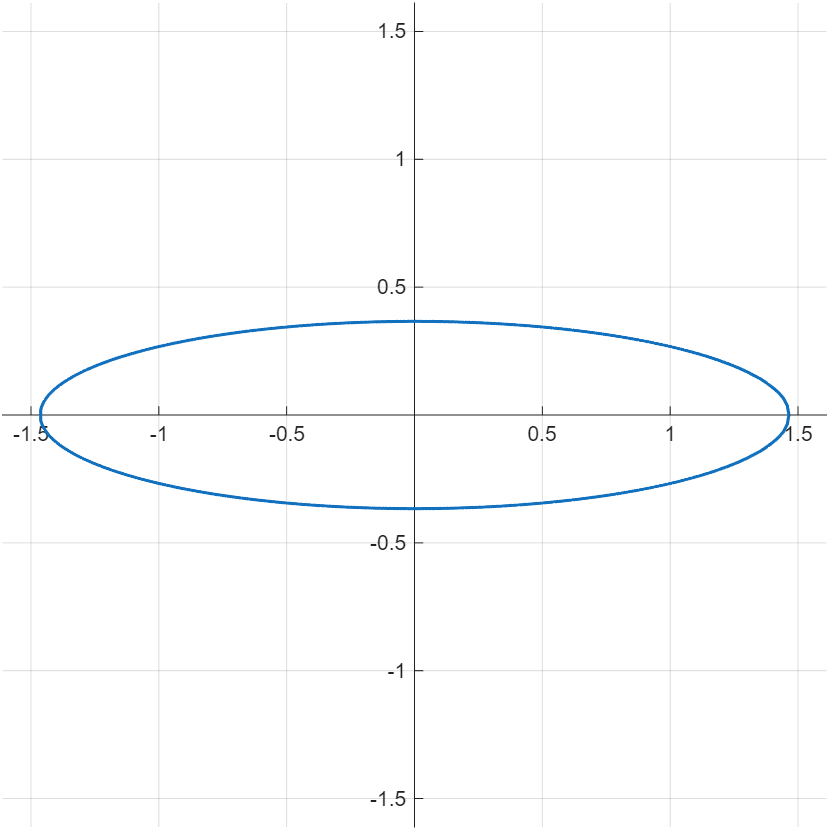}
        \caption{Recovered domain using $\pm 30000$ modes of $a$ and $\pm 30000$ modes of $h$}
    \end{subfigure}

    \caption{Numerical reconstruction of the region bounded by an ellipse. A large number of modes are required to obtain a good reconstruction of the original domain.}
    \label{ellipse_recon}
\end{figure}

We now present several numerical results of the reconstruction algorithm outline in  Section \ref{Reconalgo}. The reconstruction algorithm was
run using MATLAB R2025b on a Legion 5 (2021) with a 3.20 GHz 8 core AMD Ryzen 7 5800H processor, 16 GB 3200 MHz DDR4 RAM, and NVIDIA GeForce RTX
3060 6 GB GDDR6. Use of FINUFFT for MATLAB \cite{AHB,AHBet} greatly reduced the computation time of calculating the Fourier modes of $a$ and $h$. The code for the reconstruction algorithm can be found in \cite{code}.

\subsection{Reconstruction of an ellipse}
We reconstruct the domain $\Omega$ enclosed by an ellipse of perimeter $2\pi$  given by
$\Omega = \{ x+iy \, | \, \frac{x^{2}}{a^{2}} + \frac{y^{2}}{b^{2}} \leq 1, \frac{a}{b} = 4 \}$. We take its $201 \times 201$ truncated Hilbert
transform matrix, and feed it into the reconstruction algorithm. Figure \ref{ellipse_recon} below shows the original domain and its reconstructions.
As it is evident from the figure, a large number of modes are required to obtain a good reconstruction of the original ellipse. To obtain this level
of accuracy, $\pm 30000$ modes of the function $a$, and $\pm 30000$ modes for the function $h$ were calculated.

\subsection{Reconstruction of a domain symmetric about the $x$-axis and $y$-axis }
We consider the domain $\Omega = c\{ re^{i\theta} \, | \, 0 \leq r \leq  7 + \cos(4\theta) \, , \theta \in [0,2\pi]   \}$, where
$c \in \mathbb{R}_{> 0}$ is such that the perimeter of this domain is $2\pi$. For this reconstruction, $81 \times 81$ truncated Hilbert
transform matrix was used, $\pm 20$ modes of $a$ and $\pm 20$ modes of $h$ were used.
\begin{figure}[H]
    \centering
    \begin{subfigure}[t]{0.45\textwidth}
        \centering
        \includegraphics[width=\linewidth]{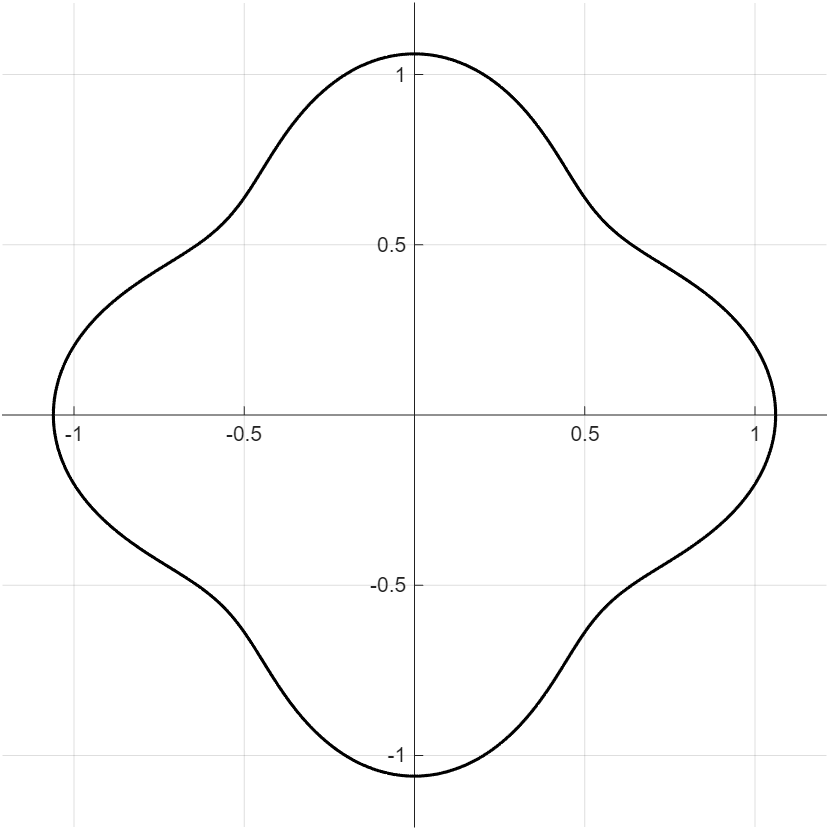}
        \caption{Original domain}
    \end{subfigure}
    \hfill
    \begin{subfigure}[t]{0.45\textwidth}
        \centering
        \includegraphics[width=\linewidth]{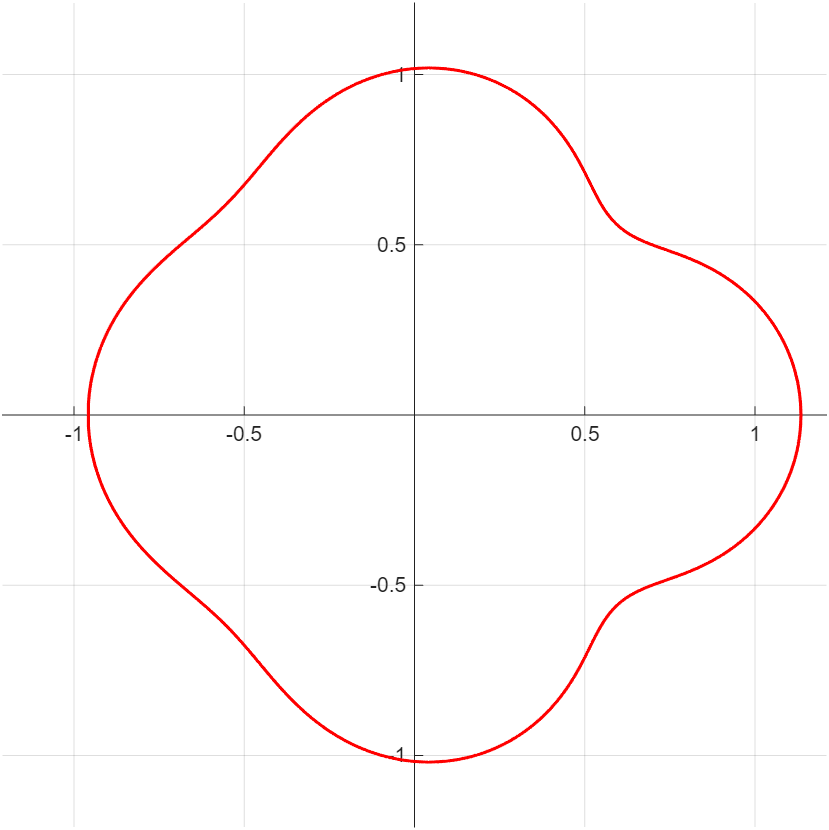}
        \caption{Recovered domain using $\pm 4$ modes of $a$ and $\pm 4$ modes of $h$}
    \end{subfigure}
    \hfill
    \begin{subfigure}[t]{0.45\textwidth}
        \centering
        \includegraphics[width=\linewidth]{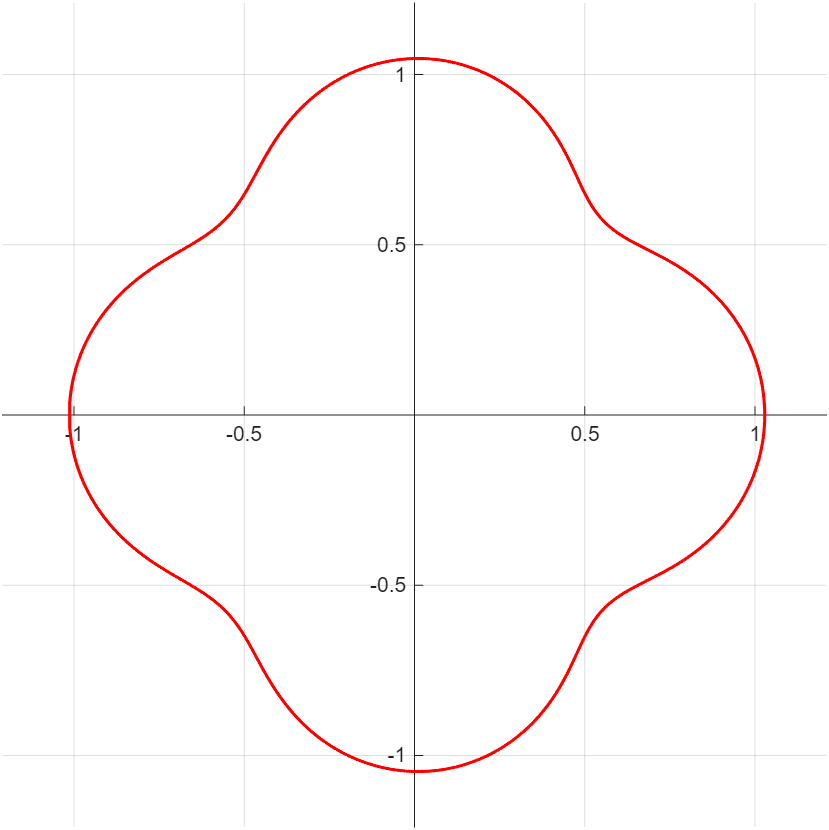}
        \caption{Recovered domain using $\pm 5$ modes of $a$ and $\pm 5$ modes of $h$}
    \end{subfigure}
    \hfill
    \begin{subfigure}[t]{0.45\textwidth}
        \centering
        \includegraphics[width=\linewidth]{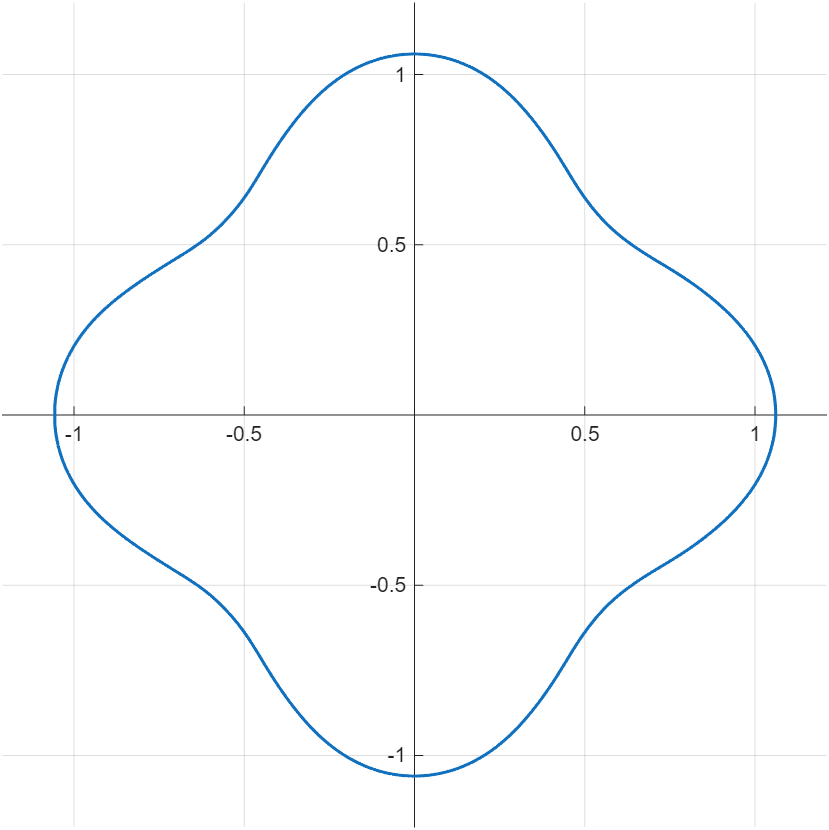}
        \caption{Recovered domain using $\pm 10$ modes of $a$ and $\pm 10$ modes of $h$}
    \end{subfigure}
    \caption{Numerical reconstruction of a domain symmetric about both the axes. Unlike the case of the elliptical domain, using far fewer
    modes achieves very good reconstruction of the original domain. }
    \label{symm_recon}
\end{figure}

\subsection{Reconstruction of a domain not symmetric about both the $x$- and $y$-axes}
We consider the domain $\Omega = c\{re^{i\theta} \in \mathbb{C}\, | \, 0 \leq r \leq  7 + \cos(4t) +1.5\sin(3t)+\sin(4t)\, , t \in \mathbb{R}\}$,
where $c \in \mathbb{R}_{> 0}$ is such that the perimeter of this domain is $2\pi$. Figure \ref{non_symm_recon} shows the reconstructions.
For this reconstruction, $201\times201$ truncated Hilbert transform matrix was used, $\pm 100$ modes of $a$ and $\pm 100$ modes of $h$ were
calculated.
\begin{figure}[H]
    \centering
    \begin{subfigure}[t]{0.45\textwidth}
        \centering
        \includegraphics[width=\linewidth]{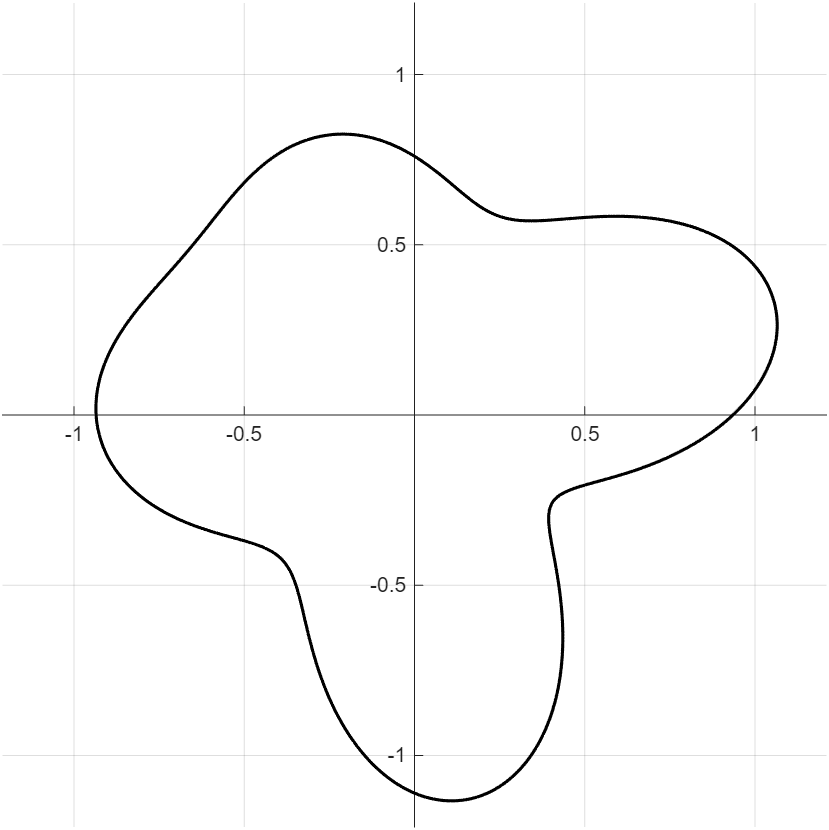}
        \caption{Original domain}
    \end{subfigure}
    \hfill
    \begin{subfigure}[t]{0.45\textwidth}
        \centering
        \includegraphics[width=\linewidth]{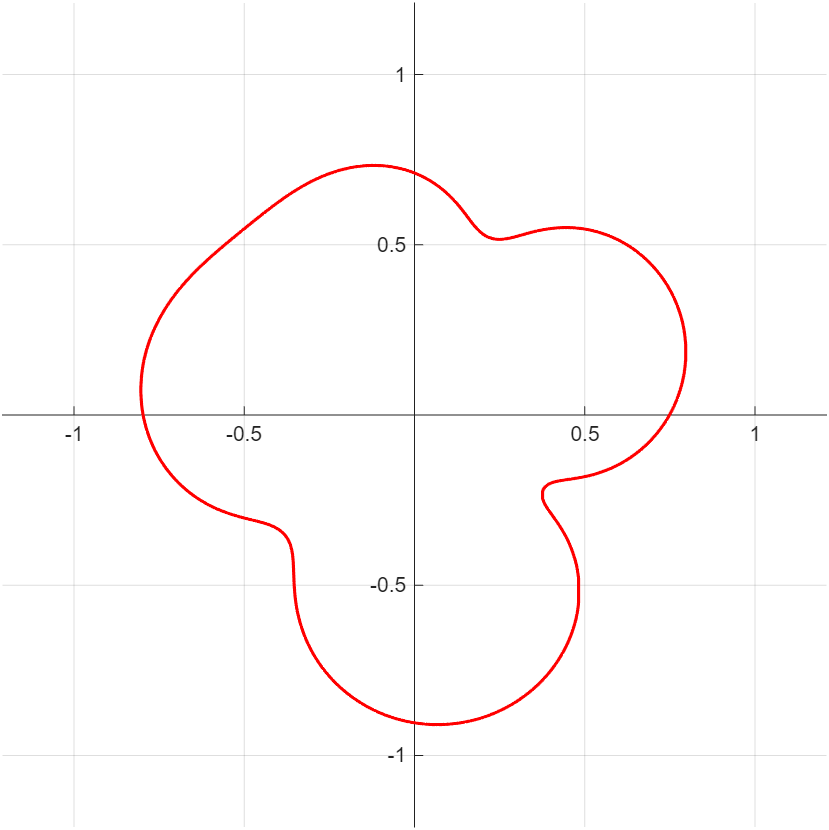}
        \caption{Recovered domain using $\pm 5$ modes of $a$ and $\pm 5$ modes of $h$}
    \end{subfigure}
    \hfill
    \begin{subfigure}[t]{0.45\textwidth}
        \centering
        \includegraphics[width=\linewidth]{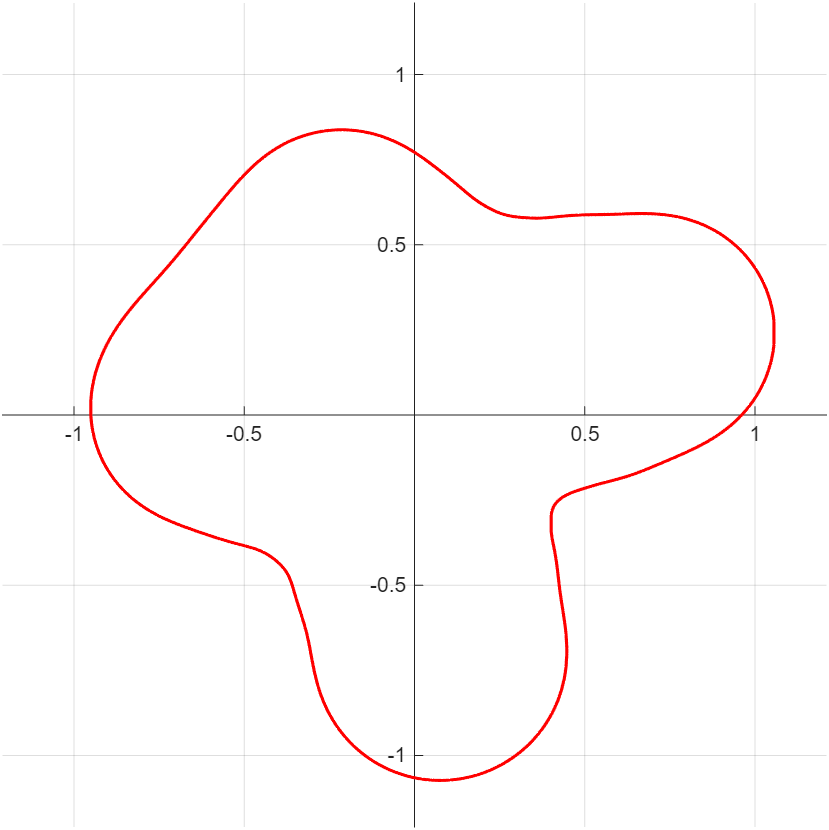}
        \caption{Recovered domain using $\pm 20$ modes of $a$ and $\pm 20$ modes of $h$}
    \end{subfigure}
    \hfill
    \begin{subfigure}[t]{0.45\textwidth}
        \centering
        \includegraphics[width=\linewidth]{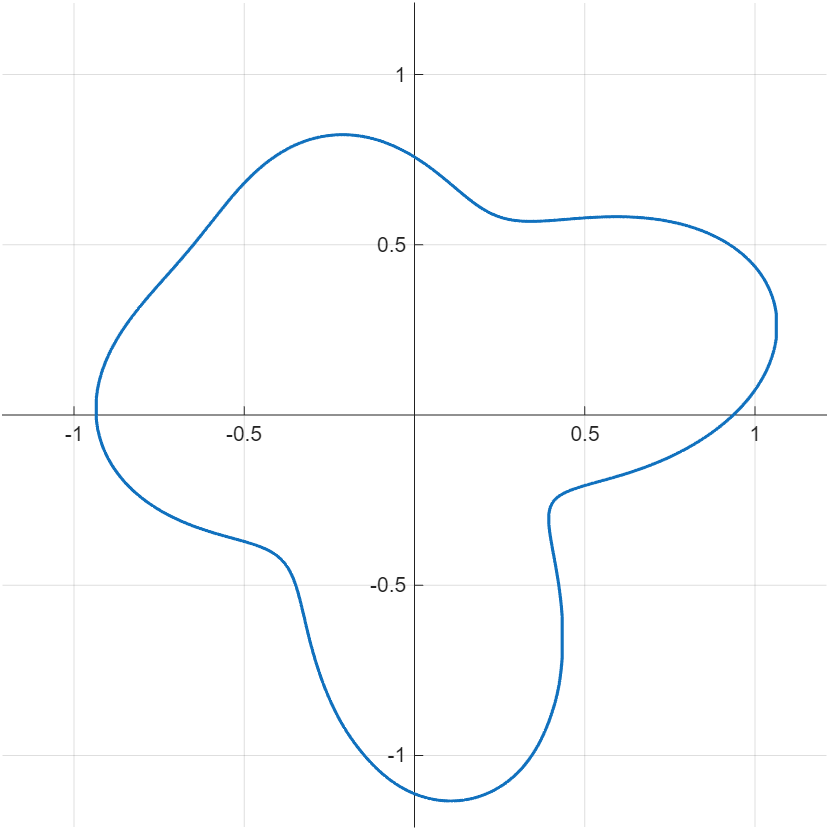}
        \caption{Recovered domain using $\pm 100$ modes of $a$ and $\pm 100$ modes of $h$}
    \end{subfigure}
    \caption{Reconstruction of a domain which isn neither symmetric about $x$ nor about $y$-axes. As with the previous case, using fewer
    modes achieves good reconstruction of the original domain. }
    \label{non_symm_recon}
\end{figure}

\begin{figure}[H]
    \centering
    \begin{subfigure}[t]{0.35\textwidth}
        \centering
        \includegraphics[width=\linewidth]{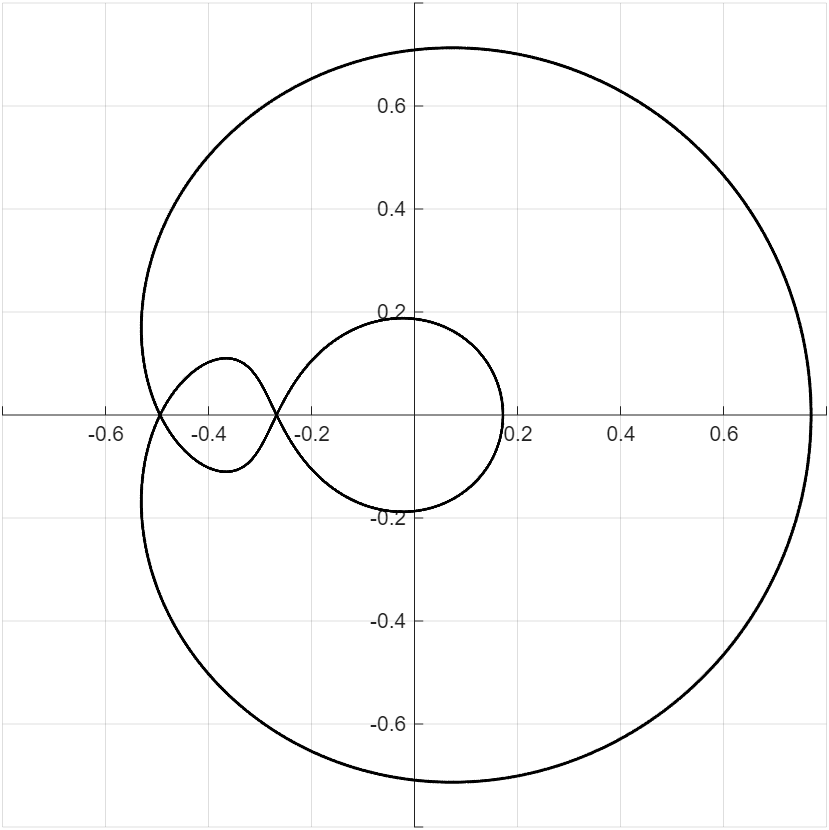}
        \caption{Original domain}
    \end{subfigure}
    \hfill
    \begin{subfigure}[t]{0.35\textwidth}
        \centering
        \includegraphics[width=\linewidth]{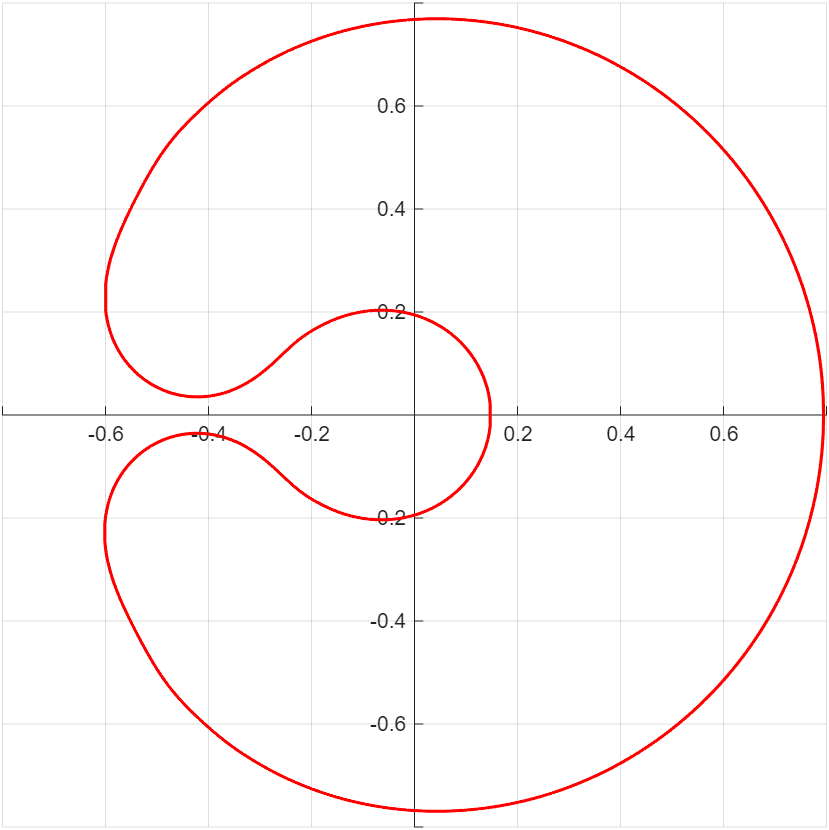}
        \caption{Recovered domain using $\pm 700$ modes of $a$ and $\pm 2500$ modes of $h$}
    \end{subfigure}
    \hfill
    \begin{subfigure}[t]{0.35\textwidth}
        \centering
        \includegraphics[width=\linewidth]{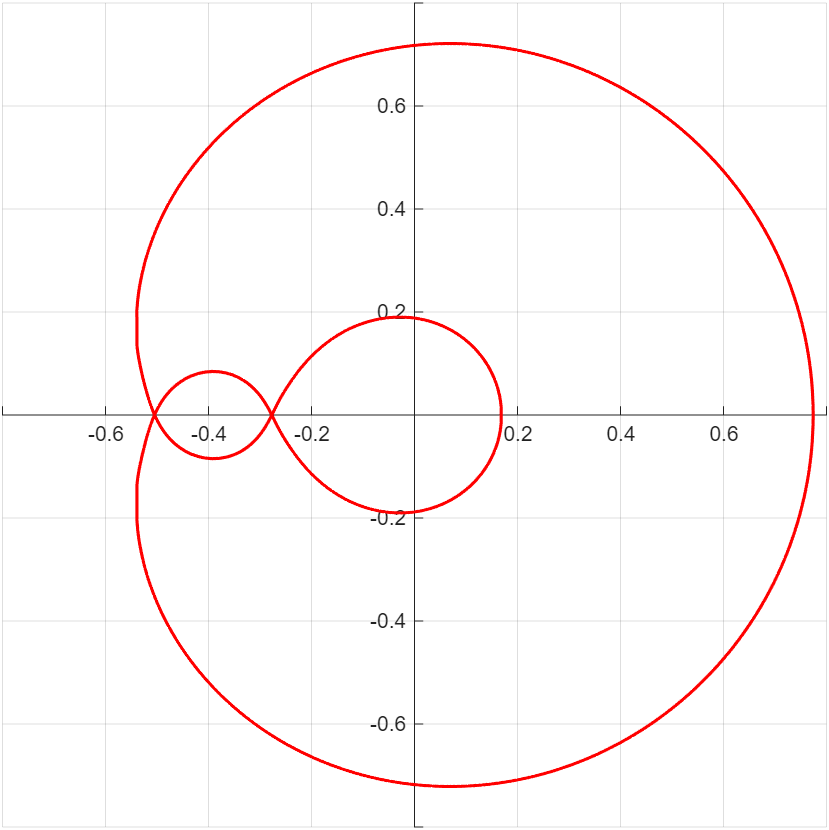}
        \caption{Recovered domain using $\pm 5000$ modes of $a$ and $\pm 15000$ modes of $h$}
    \end{subfigure}
    \hfill
    \begin{subfigure}[t]{0.35\textwidth}
        \centering
        \includegraphics[width=\linewidth]{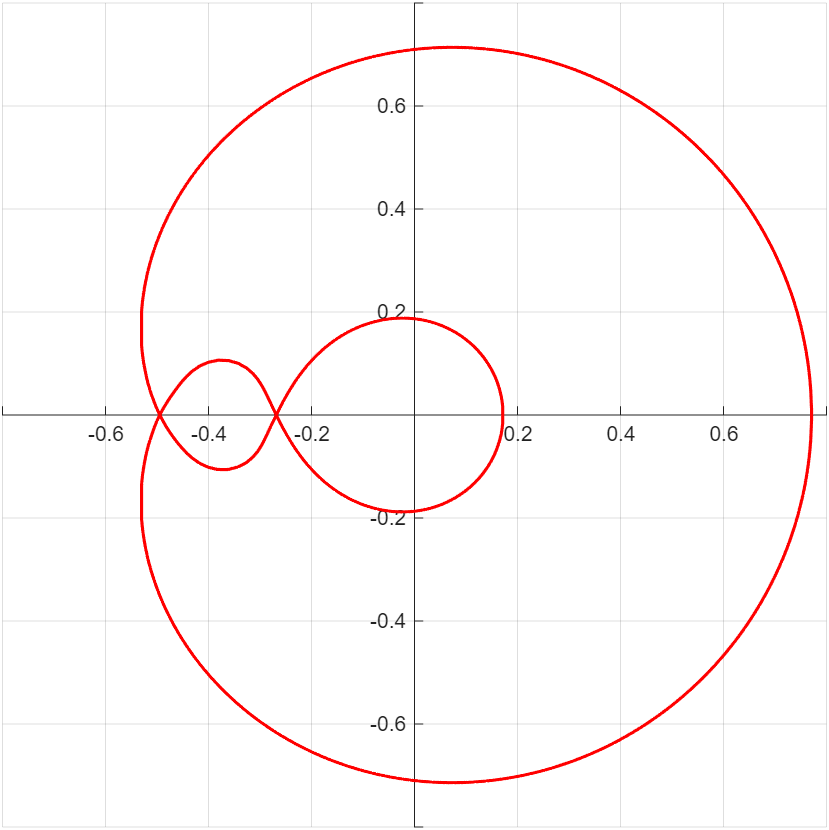}
        \caption{Recovered domain using $\pm 15000$ modes of $a$ and $\pm 25000$ modes of $h$}
    \end{subfigure}
    \hfill
    \begin{subfigure}[t]{0.4\textwidth}
        \centering
        \includegraphics[width=\linewidth]{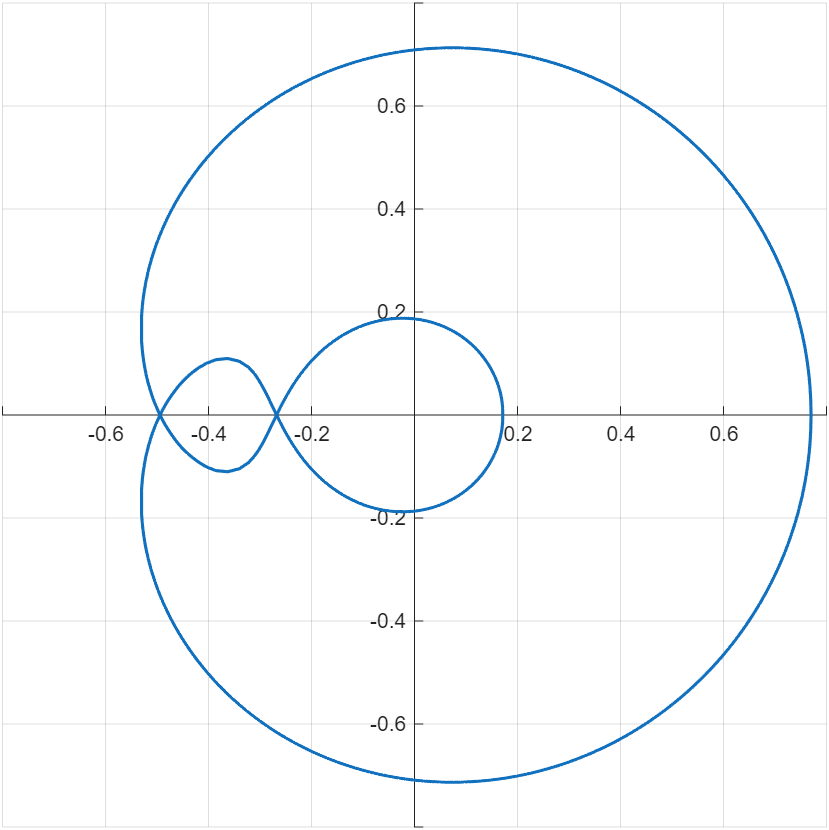}
        \caption{Recovered domain using $\pm 30000$ modes of $a$ and $\pm 50000$ modes of $h$}
    \end{subfigure}
    \caption{Reconstruction of a multi-sheet domain. A large number of modes are required to obtain a good accuracy in reconstruction.}
    \label{multisheet_recon}
\end{figure}

\subsection{Reconstruction of a multi-sheet domain}
The domain considered here is
$\Omega = c\{re^{i\theta} \, | \, e^{0.75\cos(\theta)+i(\pi+0.3) \sin(\theta)}\, , \theta \in [0,2\pi], 0\leq r \leq 1\}$, where c is chosen
such that the perimeter of the curve enclosing the domain is $2\pi$. Figure \ref{multisheet_recon} contains the results of the reconstruction.
A $301 \times 301$ truncated Hilbert transform matrix was used, $\pm 30000$ modes of $a$, and $\pm 50000$ modes of $h$ were used. It took $145$
seconds for the reconstruction algorithm to run.

{\bf Remark 1.} Recall that the function $a$ is defined by \eqref{7.48} and $h=-\ln a$. The algorithm uses truncated Fourier series \eqref{7.49} and \eqref{7.52} of these functions.
The reason it takes around $\pm 1000$ modes of $a$ and $h$ for some domains while it takes $\pm 30000$ modes of the same for
other domains (for eg. the ellipse with large $a/b$) is due to the fact that for domains like ellipse, $\Theta '$ changes slowly in some
intervals and rapidly in some other intervals, and hence $a = \frac{1}{S'}$ changes rapidly in some intervals. To resolve $a$, a large number
of modes are required.

{\bf Remark 2.} The value $\inf_{\theta \in \mathbb{S}} a(\theta)$ from \eqref{7.53} can get extremely close to zero for domains which
have parts where $\Theta '$ gets extremely close to $0$. Noise can make $\Theta '$ take negative values, which theoretically is not possible,
and hence is a potential source of unstable reconstruction.

\section{Acknowledgements}
The first author acknowledges the support of the Department of Atomic Energy, Government of India, for
PhD Fellowship.

The work of the second author was performed according to the Government research assignment
for IM SB RAS, project FWNF-2026-0026.

\end{document}